\documentclass{siamart190516}
\usepackage{graphicx}
\usepackage{amsmath}								
\usepackage{amssymb}
\usepackage{amsfonts}
\usepackage{mathtools}
\usepackage{bm}
\usepackage{bbm}
\usepackage{indentfirst}
\usepackage{enumitem}
\usepackage{algpseudocode}
\usepackage{algorithm}
\usepackage{makecell}
\usepackage{subcaption}
\usepackage{hyperref}
\usepackage{xcolor}

\newcommand{\ibold}{\mathbf{i}}
\newcommand{\jbold}{\mathbf{j}}
\newcommand{\ebold}{\bm{\mathrm{e}}}
\newcommand{\Dim}{{\scriptsize \textsf{D}}}

\newcommand{\bbR}{\mathbb{R}}
\newcommand{\bbY}{\mathbb{Y}}
\newcommand{\bbZ}{\mathbb{Z}}

\newcommand{\bmi}{\mathbf{i}}

\newcommand{\bmihalf}{\bmi + \frac 1 2 \bm{\mathrm{e}}^d}
\newcommand{\bmj}{\mathbf{j}}

\newcommand{\bmn}{\mathbf{n}}
\newcommand{\bmp}{\mathbf{p}}

\newcommand{\bmx}{\mathbf{x}}

\newcommand{\bmalpha}{\bm{\alpha}}

\newcommand{\calC}{\mathcal{C}}

\newcommand{\calF}{\mathcal{F}}
\newcommand{\calI}{\mathcal{I}}
\newcommand{\calL}{\mathcal{L}}

\newcommand{\calN}{\mathcal{N}}

\newcommand{\calS}{\mathcal{S}}
\newcommand{\calT}{\mathcal{T}}
\newcommand{\calX}{\mathcal{X}}
\newcommand{\calY}{\mathcal{Y}}
\newcommand{\cuppp}{\cup^{\perp \perp}}

\DeclareMathOperator{\diag}{diag}

\DeclareMathOperator{\bigO}{O}

\DeclareMathOperator{\intr}{int}
\newcommand{\basis}{\bm{\mathrm{e}}}
\newcommand{\rmd}{\mathrm{d}}
\newcommand{\rmC}{\mathrm{C}}
\newcommand{\rmF}{\mathrm{F}}
\newcommand{\bbone}{\mathbbm{1}}
\newcommand{\brk}[1]{\left\langle #1 \right\rangle}

\newsiamremark{remark}{Remark}

\renewcommand{\algorithmicrequire}{\textbf{Input:}}
\renewcommand{\algorithmicensure}{\textbf{Output:}}

\makeatletter
\newcommand{\fdsy@scale}{1.0}
\newcommand\fdsy@mweight@normal{Book}%
\newcommand\fdsy@mweight@small{Regular}%
\newcommand\fdsy@bweight@normal{Medium}%
\newcommand\fdsy@bweight@small{Bold}%
\DeclareFontFamily{U}{FdSymbolF}{}
\DeclareFontShape{U}{FdSymbolF}{m}{n}{
    <-7.1> s * [\fdsy@scale] FdSymbolF-\fdsy@mweight@small
    <7.1-> s * [\fdsy@scale] FdSymbolF-\fdsy@mweight@normal
}{}
\DeclareFontShape{U}{FdSymbolF}{b}{n}{
    <-7.1> s * [\fdsy@scale] FdSymbolF-\fdsy@bweight@small
    <7.1-> s * [\fdsy@scale] FdSymbolF-\fdsy@bweight@normal
}{}
\makeatother

\DeclareSymbolFont{delimiters}{U}{FdSymbolF}{m}{n}
\SetSymbolFont{delimiters}{bold}{U}{FdSymbolF}{b}{n}
\DeclareMathDelimiter{\lAngle}{\mathopen}{delimiters}{"92}{delimiters}{"92}
\DeclareMathDelimiter{\rAngle}{\mathclose}{delimiters}{"98}{delimiters}{"92}
\newcommand{\bbrk}[1]{\left\lAngle #1 \right\rAngle}

\title{A Fast Fourth-Order Cut Cell Method for Solving Elliptic Equations
in Two-Dimensional Irregular Domains}

\author{Yuke Zhu\thanks{School of Mathematical Sciences,
    Zhejiang University,
    38 Zheda Road, Hangzhou, Zhejiang Province, 310027 China.}
\and Zhixuan Li\footnotemark[1]
\and Qinghai Zhang\thanks{(Corresponding author)
    School of Mathematical Sciences,
    Zhejiang University,
    38 Zheda Road, Hangzhou, Zhejiang Province, 310027 China
    (\email{qinghai@zju.edu.cn}).}}

\headers{Cut Cell Method for Elliptic Equations}{Y. Zhu, Z. Li and Q. Zhang}

\begin{document}
\maketitle

\begin{abstract}
We propose a fast fourth-order cut cell method
for solving constant-coefficient elliptic equations
in two-dimensional irregular domains. 
In our methodology, the key to dealing with irregular domains
is the poised lattice generation (PLG) algorithm
that generates finite-volume interpolation stencils near the irregular boundary. 
We are able to derive high-order discretization of the elliptic operators
by least squares fitting over the generated stencils. 
We then design a new geometric multigrid scheme to efficiently solve the resulting linear system.
Finally, we demonstrate the accuracy and efficiency of our method
through various numerical tests in irregular domains.

\end{abstract}

\begin{keywords}
elliptic equation, fourth order, cut cell method, 
poised lattice generation, geometric multigrid.
\end{keywords}
\begin{AMS}
{65D05, 65N08}
\end{AMS}

\section{Introduction}
\label{sec:intro}

Consider the constant-coefficient elliptic equation
\begin{equation}
    a \frac{\partial^2 u}{\partial x_1^2}
    + b \frac{\partial^2 u}{\partial x_1 \partial x_2}
    + c \frac{\partial^2 u}{\partial x_2^2}
    = f \quad \text{in } \Omega,
\label{eq:ccEllipticEq}
\end{equation}
where $u : \Omega \rightarrow \bbR$ is the unknown function,
$a, b, c$ are real numbers satisfying $b^2 - 4ac < 0$,
and $\Omega$ is a open domain in $\bbR^2$.
A particular case of interest 
is Poisson's equation with $(a, b, c) = (1, 0, 1)$.
Poisson's equation plays an important role in many other mathematical formulations
such as the Helmholtz decomposition 
and the projection methods~\cite{Brown2001, Johnston2004, Liu07, Zhang2014, Zhang2016:GePUP}
for solving the incompressible Navier-Stokes equations (INSE).

A myriad of efficient and accurate numerical methods have been developed
for solving partial differential equations (PDEs) in rectangular domains. 
In many real-world applications, however, 
the problem domains are irregular with piecewise smooth boundaries.
Different numerical methods have been proposed for irregular domains,
including the notable finite element methods (FEM) on triangular meshes. 
Another approach is the cut cell method, 
in which the irregular domain is embedded into a background Cartesian mesh.
One merit of the cut cell method is the relatively inexpensive grid generation and storage.
Conventional unstructured or body-fitted grids require large computational memory 
 for storing cells, faces and a connectivity table to link them. 
The Cartesian grid, on the other hand, does not need the storage of any connectivity table
 as they can be mapped by simple indices.
As a result, the computational time is also reduced.
Another merit is that mature techniques
can be borrowed from the finite difference and finite volume community.
Examples are multigrid algorithms~\cite{Briggs:A_Multigrid_Tutorial} for elliptic equations,
shock-capturing schemes~\cite{Shu1988, Liu1994} for hyperbolic systems
and high-order conservative schemes~\cite{Morinishi1998} for incompressible flows.

It is worth mentioning a class of finite difference Cartesian grid methods
for the related ``interface problems''.
The Immersed Interface Method (IIM)~\cite{LeVeque1994, Li1998, Li2003, Linnick2005}
modifies the discretization stencil
to incorporate jump conditions on the interface.
The Ghost Fluid Method (GFM)~\cite{Liu2000, Gibou2002, Liu2003, LiuMGFM2023}
introduces ghost values as smooth extensions of the physical variables across the interface,
so that conventional finite difference formulas can be employed.



For problems involving a single-phase fluid, the cut cell method 
(also known as Cartesian grid method and embedded boundary (EB) method) is preferred. 
Previously, second-order cut cell methods 
have been developed for Poisson's equation~\cite{Johansen1998, Schwartz2006},
heat equation~\cite{McCorquodale2001, Schwartz2006}
and the Navier-Stokes equations~\cite{Kirkpatrick2003, Trebotich15}.
Recently, a fourth-order EB method for Poisson's equation~\cite{Devendran2017}
was developed by Devendran et al. 
However, a number of issues presented in the cut cell method
are only partially resolved, or even unresolved:
\begin{enumerate}[label={(Q-\arabic*)}, leftmargin=*]
\item
Given that the problem domains can be arbitrarily complex in geometry and topology,
is there an accurate and efficient representation of such domains?
\label{q:representationOfDomains}

\item
For domains with arbitrarily complex boundaries,
the generation of cut cells and the stability issues 
arising from the small cells within them are intricate.
Can we design a cutting algorithm to produce control volumes suitable for discretization?
\label{q:mergingAlgorithm}

\item
Previously, the selection of stencils 
depended very much on both the order of discretization
and the specific form of the differential operator.
In these methods,
obtaining a high-order discretization of the cross derivative term
in complex geometries can be tricky.
So can we generate stencils that adapt automatically to complex geometries,
and meanwhile minimize their coupling with the high-order discretization
of the differential operator?
\label{q:stencilGeneration}

\item
Having discretized the elliptic equation,
can we solve the resulting linear systems efficiently
to produce fourth-order accurate solutions?
\label{q:efficiency}
\end{enumerate}


In this paper, we give positive answers to all the above questions
by proposing a fast fourth-order cut cell method
for solving constant-coefficient elliptic equations
in two-dimensional irregular domains.
%

To answer~\ref{q:representationOfDomains},
we make use of the theory of Yin space~\cite{Zhang2020:YinSets}.
The member of Yin space, called Yin set, is the mathematical model
for physically meaningful regions in the plane.
Every Yin set has a simple representation that facilitates geometric and topological queries.
In this work, we will formulate the computational domains
and the cut cells in terms of Yin sets.
Our answer to~\ref{q:mergingAlgorithm}
relies on the Boolean operations equipped by the Yin space;
see Theorem~\ref{thm:booleanAlgebra}.
We propose a systematic algorithm for 
generating cut cells and merging the small cells 
with volume fraction below a user-specified threshold
in the computational domain.
The key to answering~\ref{q:stencilGeneration}
is the poised lattice generation (PLG) algorithm~\cite{Zhang:PLG}.
Given a prescribed region of $\bbZ^\Dim$,  
the PLG algorithm generates poised lattices on which multivariate polynomial interpolation is unisolvent.
It has been successfully employed to obtain fourth- and sixth-order discretization
of the convection operator in complex geometries.
Based on the interpolation stencils, 
we demonstrate a general approach for
deriving high-order approximations to linear differential operators,
using least squares fitting by complete multivariate polynomials. 
This approach handles different boundary conditions
and can be extended to nonlinear differential operators as well.
To answer~\ref{q:efficiency},
we design a new geometric multigrid algorithm for efficiently solving the resulting discrete linear system.
We modify the conventional multigrid components to
account for irregular domains and high-order discretization.
We show by analysis and by numerical tests that
the linear systems are solved with optimal complexity. 

The rest of this paper is organized as follows. 
In Section~\ref{sec:preliminaries}
we collect preliminary concepts, 
with focus on the PLG algorithm 
in finite difference formulation. 
In Section~\ref{sec:discretization},
we propose a fourth-order cut cell method
for constant-coefficient elliptic equations in irregular domains. 
Specifically,
we discuss the cut-cell generation in Section~\ref{sec:grid},
the standard finite-volume approximation in Section~\ref{sec:FVAppro},
the generation of finite-volume stencils near irregular boundaries,
and the high-order discretization of the differential operators 
in Section~\ref{sec:discretizationOfSpatialOp}.
We then present a multigrid algorithm for efficiently solving the discrete linear system
and analyze its complexity in Section~\ref{sec:multigrid}.
In Section~\ref{sec:Tests},
we perform numerical tests in various irregular domains
to demonstrate the accuracy and efficiency of our method. 
Finally, we conclude this paper with some future research prospects
in Section~\ref{sec:conclusions}.

\section{Preliminaries}
\label{sec:preliminaries}

In this section, we collect preliminary concepts that are necessary for
the development of the proposed numerical method in this paper.

\subsection{Yin space}

To answer the need for an accurate and efficient representation of the irregular domains,
we review a topological space for modeling physically meaningful regions.

Let $\calX$ be a topological space, 
and $\calS, \calT$ be two subsets of $\calX$.
Denote by $\calS^-$ the closure of $\calS$,
and by $\calS^\perp$ the exterior of $\calS$,
i.e.~the interior of the complement of $\calS$. 
A subset $\calS$ is \emph{regular open} if it coincides with the interior of its closure.
Denote by $\cuppp$ the regularized union operation, 
i.e.~$\calS \cuppp \calT := \left(\calS \cup \calT \right)^{\perp \perp}$. 
It can be shown~\cite[\S 10]{Givant:BooleanAlgebra} that
the class of all regular open subsets of $\calX$
is closed under the regularized union. 
For $\calX = \bbR^2$, 
a subset $\calS \subset \bbR^2$ is \emph{semianalytic}
if there exist a finite number of analytic functions
$g_i : \bbR^2 \rightarrow \bbR$
such that $\calS$ is in the universe of a finite Boolean algebra
formed from the sets
\begin{equation}
\calX_i = \left\{ \bmx \in \bbR^2 : g_i(\bmx) \ge 0 \right\}. 
\end{equation}
In particular, a semianalytic set is \emph{semialgebraic}
if all the $g_i$'s are polynomials. 

\begin{definition}[Yin Space~\cite{Zhang2020:YinSets}]
A \emph{Yin set} $\calY \subset \bbR^2$ is a regular open semianalytic set 
whose boundary is bounded. 
The class of all such Yin sets form the Yin space $\bbY$. 
\end{definition}

In the above definition, 
regularity captures the physical meaningfulness of the fluid regions,
openness guarantees a unique boundary representation, 
and semianalyticity implies finite representation.

\begin{definition}
The subspace $\bbY_{\mathrm{c}}$ of $\bbY$
consists of Yin sets whose boundaries are captured by cubic splines.
\end{definition}

\begin{theorem}[Zhang and Li~\cite{Zhang2020:YinSets}]
$\left(\bbY, \cuppp, \cap, ^\perp, \emptyset, \bbR^2 \right)$ is a Boolean algebra. 
\label{thm:booleanAlgebra}
\end{theorem}

\begin{corollary}
$\bbY_{\mathrm{c}}$ a sub-algebra of $\bbY$. 
\end{corollary}

Efficient algorithms have been developed in~\cite[Section 4]{Zhang2020:YinSets}
to implement the Boolean operations in Theorem~\ref{thm:booleanAlgebra}.
In this work, they will be used for cut cell generation and merging.

Consider the representation of a Yin set.
If $\gamma$ is an oriented Jordan curve,
denote by $\intr(\gamma)$ the complement of $\gamma$
that always lies at the left of an observer who traverses $\gamma$
according to its orientation.
It is shown in~\cite[Corollary 3.13]{Zhang2020:YinSets}
that every Yin set $\calY \neq \emptyset, \bbR^2$ can be uniquely expressed as
\begin{equation}
 \calY = \bigcup_{j} {}^{\perp\perp} \bigcap_{i} \intr(\gamma_{j, i}),
\end{equation}
where for each fixed $j$ the collection of oriented Jordan curves
$\{\gamma_{j,i}\}_{i}$ is pairwise almost disjoint
and corresponds to a connected component of $\calY$.
The whole collection $\{\gamma_{j,i}\}$ can be further arranged in a fashion
such that the topological information (such as Betti numbers)
can be easily extracted within constant time;
see~\cite[Definition 3.17]{Zhang2020:YinSets}.


\subsection{Poised lattice generation}
\label{sec:PLG}

Traditional finite difference (FD) methods
have limited applications in irregular domains,
as they often replace the spatial derivatives in partial differential equations 
by one-dimensional FD formulas or their tensor-product counterparts.
To overcome this limitation,
the PLG algorithm~\cite{Zhang:PLG} was proposed 
to generate poised lattices in complex geometries.
Once the interpolation lattices have been generated,
high-order discretization of the differential operators can be obtained
via multivariate polynomial fitting.

In this section we give a brief review of the PLG problem.
We begin with the relevant notions in multivariate interpolation.
\begin{definition}[Lagrange interpolation problem (LIP)]
Denote by $\Pi_n^{\Dim}$ the vector space of 
all $\Dim$-variate polynomials of degree no more than $n$
with real coefficients. 
Given a finite number of points $\bmx_1, \bmx_2, \cdots, \bmx_N \in \bbR^\Dim$
and the same number of data $f_1, f_2, \cdots, f_N \in \bbR$, 
the \emph{Lagrange interpolation problem} seeks a polynomial $f \in \Pi_n^\Dim$ such that
\begin{equation}
f(\bmx_j) = f_j, \quad \forall j=1, 2, \cdots, N, 
\label{eq:finiteDiffLIP}
\end{equation}
where $\Pi_n^\Dim$ is the \emph{interpolation space} 
and $\bmx_j$'s the \emph{interpolation sites}. 
\label{def:finiteDiffLIP}
\end{definition}

The sites $\left\{ \bmx_j \right\}_{j=1}^N$ are said to be \emph{poised} in $\Pi_n^\Dim$
if there exists a unique $f \in \Pi_n^\Dim$ satisfying (\ref{eq:finiteDiffLIP})
for any given data $\left\{ f_j \right\}_{j=1}^N$. 
Suppose $\left\{ \phi_j \right\}$ is a basis of $\Pi_n^\Dim$.
It is easy to see that the interpolation sites are poised
if and only if $N = \dim \Pi_n^\Dim = \binom{n+\Dim}{n}$
and the $N \times N$ \emph{sample matrix}
\begin{equation}
M(\left\{ \phi_j \right\} ; \left\{ \bmx_k \right\}) = \Big[ M_{jk} \Big] := \Big[ \phi_j(\bmx_k) \Big]
\quad (1 \le j, k \le N)
\label{eq:sampleMatrix}
\end{equation}
is non-singular. 

\begin{definition}[PLG in $\mathbb{Z}^\Dim$~\cite{Zhang:PLG}]
\label{def:PLG}
    Given a finite set of feasible nodes $K\in \mathbb{Z}^\text{D}$, 
	 a starting point $\mathbf{q}\in K$, and a degree $n\in \mathbb{Z}^+$,
	 the \emph{problem of poised lattice generation} is to
	 choose a lattice $\mathcal{T}\subset K$ such that 
	 $\mathcal{T}$ is poised in $\Pi_n^\Dim$ and $\#\mathcal{T} = \dim \Pi_n^\Dim$.
\end{definition}

In \cite{Zhang:PLG}, we proposed a novel and efficient PLG algorithm to 
 solve the PLG problem via an interdisciplinary research of approximation theory,
 abstract algebra, and artificial intelligence.
The stencils are based on the transformation of the principal lattice.
In this work, we directly apply the algorithm to our finite volume discretization.
The starting point $\mathbf{q}$ corresponds to the cell index 
at which the spatial operators are discretized,
 and the feasible nodes $K$ formed by collecting the indices of nearby cut cells.
 
In general cases it is still unclear that
 whether a direct application of the PLG algorithm
 will generate poised stencils in finite volume formulation.
Finding a strict proof of poisedness 
 will be a topic of our future research.
Nevertheless, the poisedness in finite volume formulation is supported
 by the extensive numerical evidences.
For example, the condition numbers of the $(\Dim,n) = (2,4)$ lattices used in this paper
generally do not exceed $10^4$ in all problems of interest.

\subsection{Least squares problems}
\label{sec:LSP}

To derive the approximation to the differential operators,
we need the following results on weighted least squares fitting.

\begin{definition}
For a symmetric positive definite matrix $W \in \bbR^{n \times n}$, 
the \emph{$W$-inner product} is
\begin{equation}
\langle u, v \rangle_W := u^T W v
\end{equation}
and the \emph{$W$-norm} is
\begin{equation}
\Vert u \Vert_W := \langle u, u \rangle_W^{1/2}
\end{equation}
for $u, v \in \bbR^n$. 
\end{definition}

\begin{proposition}
Let $A \in \bbR^{m \times n}$ be a matrix with full column rank,  and let 
$W \in \bbR^{m \times m}$ be a symmetric positive definite matrix. 
\begin{enumerate}[label={(\roman*)}, nosep, leftmargin=*]
\item For any $b \in \bbR^m$, the optimization problem
\begin{equation}
\min_{x \in \bbR^n} \left\Vert Ax - b \right\Vert_{W}
\end{equation}
admits a unique solution $x_{\mathrm{LS}} = \left(A^T W A\right)^{-1} A^T W b$. 
\item For any $d \in \bbR^n$, the constrained optimization problem
\begin{equation}
\min_{x \in \bbR^m} \left\Vert x \right\Vert_{W^{-1}} \quad \mathrm{s.t.} \quad A^T x=d
\end{equation}
admits a unique solution $x_{\mathrm{NM}} = W A \left(A^T W A\right)^{-1} d$. 
\end{enumerate}
\label{prop:leastSquares}
\end{proposition}
\begin{proof}
See~\cite[\S 5.3, \S 5.6, \S 6.1]{Golub}.
\end{proof}

\subsection{Numerical cubature}
\label{sec:NC}

In finite volume discretization,
we need to evaluate the volume integral
of a given function $f(x, y)$ over a control volume
$\calC \in \bbY_{\mathrm{c}}$. 
For this purpose, 
we invoke the Green's formula to convert
\begin{equation}
\iint_{\calC}  f(x, y) \rmd x \rmd y
= \oint_{\partial \calC} F(x, y) \rmd y
\label{eq:volumeIntegral}
\end{equation}
into a line integral, 
where $\partial \calC$ is the boundary of $\calC$ and 
$F(x, y) = \int_{\xi_0}^{x} f(\xi, y) \rmd \xi$
is a primitive of $f(x, y)$ with respect to $x$,
with $\xi_0$ an arbitrary fixed real number.
Suppose $\left( x(t), y(t) \right) \left( t \in [0,1] \right)$
is a smooth parametrization of $\partial \calC$. 
Then we can write 
\begin{subequations}
\begin{align}
 \oint_{\partial \calC} F(x, y) \rmd y &= \int_0^1 F(x(t), y(t)) \dot{y}(t) \rmd t \\
&=  \int_0^1\dot{y}(t) \int_{\xi_0}^{x(t)} f(\xi, y(t)) \rmd \xi  \rmd t, 
\label{eq:iteratedIntegral}
\end{align}
\end{subequations}
into a two-fold iterated integral. 
The integral (\ref{eq:iteratedIntegral}) 
can be evaluated using one-dimensional numerical quadratures 
such as Gauss-Legendre quadrature.
If the boundary $\partial \calC$ is only piecewise smooth, 
we simply apply (\ref{eq:iteratedIntegral}) to each piece and sum up the results.
If the integrand $f(x, y)$ is a bivariate polynomial, 
then (\ref{eq:iteratedIntegral}) expands to a (possibly very high-order) polynomial of $t$
whose coefficients depend on the parametrization $\left( x(t), y(t) \right)$ and the integrand $f(x, y)$. 
This polynomial can then be evaluated to machine precision without quadratures.

Although $\xi_0$ can be arbitrary, 
it should be selected to minimize round-off errors. 
A typical choice is the center of the bounding box of $\calC$.
See~\cite{Sommariva} for a detailed error analysis
on using Gauss-Legendre quadrature to evaluate the cubature formula (\ref{eq:iteratedIntegral}).

\section{Spatial discretization}
\label{sec:discretization}

\subsection{Cut-cell generation}
\label{sec:grid}
 
Let $R$ be a rectangular region containing $\Omega \in \bbY_\mathrm{c}$ and
partitioned uniformly into a collection of rectangular cells
\begin{equation}
\rmC_{\bmi} := \Big( \bmi h, (\bmi + \bbone) h \Big), 
\end{equation}
where $h$ is the uniform spatial step size\footnote{
The methodology in this work applies to non-uniform step sizes as well. 
Using open intervals instead of closed intervals 
is for consistency with the representation of Yin sets.
}, 
$\bmi \in \bbZ^{\Dim}$ a multi-index, 
and $\bbone \in \bbZ^{\Dim}$ the multi-index with all its components equal to one.
The higher face of the cell $\rmC_{\bmi}$ in dimension $d$ is denoted by
\begin{equation}
\rmF_{\bmihalf} := \Big( (\bmi + \basis^d) h, (\bmi + \bbone) h \Big), 
\end{equation}
where $\basis^d \in \bbZ^\Dim$ is the multi-index 
with 1 as its $d$th component ($1\le d \le \Dim$) and 0 otherwise. 
The cut cell method embeds the irregular domain $\Omega$ into the rectangular grid $R$. 
We define the cut cells by
\begin{equation}
\calC_{\bmi} := \rmC_\bmi \cap \Omega, 
\label{eq:cutCell}
\end{equation}
the cut faces by
\begin{equation}
\calF_{\bmihalf} := \rmF_{\bmihalf} \cap \Omega, 
\label{eq:cutFace}
\end{equation}
and the portion of domain boundary contained in a cut cell by
\begin{equation}
\calS_{\bmi} := \rmC_{\bmi} \cap \partial \Omega. 
\label{eq:cutBdry}
\end{equation}

A cut cell is said to be an \emph{empty} cell if $\mathcal{C}_\ibold = \emptyset$,
 a \emph{pure cell} if $\mathcal{C}_\ibold = \rmC_\ibold$, 
 or an \emph{interface cell} otherwise.
In practice, cut cells may be arbitrarily small 
 and the elliptic operator will be very ill-conditioned.
A number of approaches have been proposed to address this problem: 
 for example, the application of cell merging~\cite{Ji2010},
 the utilization of redistribution~\cite{Almgren1994}, 
 or the formulation of special differencing schemes for the discretization~\cite{berger1991}. 
Of these various approaches, we choose to use a simple cell-merging methodology 
 to circumvent the small-cell problem.
 
\begin{algorithm}[htbp]
  \caption{Cut-cell generation and merging.}
  \label{alg:merging}
  \begin{algorithmic}[1]
    \Require a computational domain $\Omega\in\mathbb{Y}_c$; a Cartesian grid with step size $h$; 
             a user-specified threshold $\theta\in(0,1)$.
    \Ensure a set of cut cells $\left\{\mathcal{C}_\ibold\right\}_{\ibold\in\mathbb{Z}^{\Dim}}$.
    \renewcommand{\algorithmicrequire}{\textbf{Precondition:}}
    \renewcommand{\algorithmicensure}{\textbf{Postcondition:}}
    \Ensure $\bigcup^{\bot\bot}_{\ibold} \mathcal{C}_\ibold = \Omega$ 
            and the volume of each cut cell is not less than $\theta h^2$.
    \renewcommand{\algorithmicrequire}{\textbf{Input:}}
    \renewcommand{\algorithmicensure}{\textbf{Output:}}
    
    \State
    $\left\{\mathcal{C}_\ibold\right\}_{\ibold\in\mathbb{Z}^{\Dim}} \leftarrow \left\{ \mathrm{C}_\ibold \cap \Omega \right\}_{\ibold\in\mathbb{Z}^{\Dim}}$ : the set of cut cells obtained by embedding $\Omega$ into the Cartesian grid.
    \State
    Preprocess the cut cells $\mathcal{C}_\ibold = \left( \bigcup^{\bot\bot} \mathcal{C}_\ibold^k \right)_{k=1}^{n_c}$ with $n_c (n\geq 2)$ connected components by
    \begin{equation*}
      \mathcal{C}_\ibold \leftarrow \mathcal{C}_\ibold^m, 
      \quad \mathcal{C}_\jbold \leftarrow \mathcal{C}_\jbold \cup^{\bot\bot} \mathcal{C}_\ibold^k,
    \end{equation*}
    where $\mathcal{C}_\ibold^m = \max\{ \Vert\mathcal{C}_\ibold^1\Vert, \Vert\mathcal{C}_\ibold^2\Vert, \cdots, \Vert\mathcal{C}_\ibold^{n_c}\Vert\}$ and $\mathcal{C}_\jbold$ is the largest neighboring cell of $\mathcal{C}_\ibold^k$ for $k\neq m$.
    \For{\textbf{each} interface cell $\ibold$ satisfies $\Vert \mathcal{C}_\ibold \Vert \leq \theta h^\Dim$}
    \State
    Locate the no-empty neighboring cells along the normal vector direction to the cut face of $\mathcal{C}_\ibold$:
    \begin{equation*}
    	M = \{ \jbold\in\mathbb{Z}^\Dim : \Vert \jbold-\ibold\Vert_1 \leq 1 \;\;\text{and}\;\; \Vert \mathcal{C}_\jbold \Vert \neq 0\}.
    \end{equation*}
    \State
    Select $\jbold^+ = \arg\max_{\jbold\in M}  \Vert \mathcal{C}_\ibold \cap \mathcal{C}_\jbold \Vert$.
    \State   
    Set $\mathcal{C}_{\jbold^+} \leftarrow \mathcal{C}_{\jbold^+}\bigcup^{\bot\bot}\mathcal{C}_\ibold $ and $\mathcal{C}_\ibold \leftarrow \emptyset$.
    \EndFor

\end{algorithmic}	
\end{algorithm}

We use Algorithm~\ref{alg:merging} for 
 generating cut cells and handling the small-cell problem.
In Algorithm~\ref{alg:merging}, 
 to determine these initial cut cells in line 1 in a robust manner,  
Boolean operations are used.
A detailed description of the implementation 
 of these Boolean operations can be found in~\cite{Zhang2020:YinSets}. 
After the Boolean operations, if a cut cell has multiple connected components,
 we first merge all components except the one with the largest volume into its neighbors.
This situation is common when the boundary is a sharp interface,
 even if the grid size is very small.
Lines 3-7 describe a simple cell-merging procedure.
The cell to merge with is selected by finding the neighboring cells 
 in the direction of the normal vector to the cut face
 and choose the one with the largest common face.
It should be noted that we do not store any variable values for the small cells. 
More specifically, variable values are only stored for the merged cell, 
 consisting of the small cell and its neighbor cell. 
An example is shown in Figure~\ref{fig:gridsExample}.

\begin{figure}[htbp]
	\center
	\includegraphics[width=0.4\linewidth]{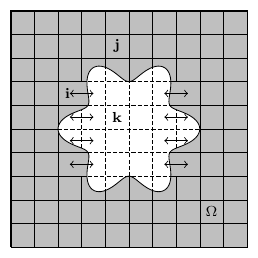}
	\caption{A Cartesian grid in the cut cell method.
             Here small cells with volume fraction below 0.3 have been merged.
             The cut cells $\mathbf{i}$, $\mathbf{j}$, and $\mathbf{k}$ are interface, 
             pure and empty cells, respectively.
             The merged cut cells are linked by the symbol ``$\leftrightarrow$".}
    \label{fig:gridsExample}
\end{figure}

\subsection{Finite volume approximation}
\label{sec:FVAppro}

Define the averaged $\phi$ over cell $\ibold$ by
\begin{equation}
\brk{\phi}_{\bmi} := \frac {1} {\lVert \calC_\bmi \rVert} \int_{\calC_{\bmi}} \phi(\bmx) \ \rmd \bmx,
\end{equation}
the averaged $\phi$ over face $\bmihalf$ by
\begin{equation}
\brk{\phi}_{\bmihalf} := 
\frac {1} {\lVert \calF_{\bmihalf} \rVert} \int_{\calF_{\bmihalf}} \phi(\bmx) \ \rmd \bmx,
\end{equation}
and averaged $\phi$ over cell boundary $\mathcal{S}_\ibold$ by
\begin{equation}
\bbrk{\phi}_{\bmi} := \frac {1} {\Vert \calS_\bmi \Vert}
\int_{\calS_\bmi} \phi(\bmx) \ \rmd \bmx.
\end{equation}
On cut cells sufficiently far from irregular boundaries,
a routine Taylor expansion justifies 
the following fourth-order difference formulas
for the discrete gradient and discrete Laplacian:
\begin{align}
\label{eq:stdGrad}
\brk{\frac {\partial \phi} {\partial x_d}}_{\bmi} &= \frac {1} {12 h} \Big(
- \brk{\phi}_{\bmi + 2 \basis^d} 
+ 8 \brk{\phi}_{\bmi + \basis^d}
- 8 \brk{\phi}_{\bmi - \basis^d}
+ \brk{\phi}_{\bmi - 2 \basis^d} \Big)
+ \bigO(h^4), \\
\label{eq:stdLap}
\brk{\Delta \phi}_{\bmi} 
&= \frac {1} {12h^2} \sum_d \Big( -\brk{\phi}_{\bmi + 2\basis^d} + 16 \brk{\phi}_{\bmi + \basis^d} 
-30 \brk{\phi}_{\bmi} + 16 \brk{\phi}_{\bmi - \basis^d} + \brk{\phi}_{\bmi - 2\basis^d} \Big) + \bigO(h^4). 
\end{align} 
For the cross derivative $\frac {\partial^2 \phi} {\partial x_i \partial x_j} (i \neq j)$, 
a fourth-order discretization is obtained
from the iterated applications of (\ref{eq:stdGrad}) in the $i$th and $j$th directions.
The stencils of the Laplacian operator and the cross operator are thus
\begin{equation}
\mathcal{S}_\mathrm{Lap}(\ibold) = \{ \ibold+m\ebold^d : m = 0, \pm 1, \pm 2, d = 1, 2.\},
\end{equation}
and
\begin{equation}
\mathcal{S}_\mathrm{Cro}(\ibold) = \{ \ibold + m_1\ebold^1 + m_2\ebold^2 : m_1, m_2 = \pm 1, \pm2.\}.
\end{equation}

A cell $\ibold$ will be called a \emph{regular cell} 
 if each cell $\jbold\in \mathcal{S}_\mathrm{Lap}(\ibold) \cup \mathcal{S}_\mathrm{Cro}(\ibold)$
 is a pure cell.
A non-empty cell is either regular or \emph{irregular}. 
The region covered by regular and irregular cells 
 are called the \emph{regular region} and the \emph{irregular region}, respectively.
For a regular cell $\ibold$,
 the elliptic operator can be evaluated using Equations \eqref{eq:stdGrad} and \eqref{eq:stdLap}.
For an irregular cell $\ibold$,
 the elliptic operator can be evaluated by 
 first generating poised lattices near $\mathcal{C}_\ibold$ (Section \ref{sec:PLG}),
 then locally reconstructing a multivariate polynomial (Section \ref{sec:LSP} and \ref{sec:discretizationOfSpatialOp}),
 and finally evaluating the multi-dimensional quadratures (Section \ref{sec:NC}).


\subsection{Discretization of the spatial operators}
\label{sec:discretizationOfSpatialOp}

In this section, we consider the discretization of 
a linear differential operator $\calL$ of order $q$
near the irregular domain boundary.
Although we are mainly concerned with the second-order elliptic operator
\begin{equation}
 \label{eq:theEllipticOperator}
 u \mapsto a \frac {\partial^2 u} {\partial x_1^2}
 + b \frac {\partial^2 u} {\partial x_1 \partial x_2}
 + c \frac {\partial^2 u} {\partial x_2^2},
\end{equation}
the methodology applies to other linear operators as well.
Denote by $\calN$ the operator of the boundary condition:
for example, 
$\calN = \mathbf{I}_\text{d}$ for the Dirichlet condition, 
$\calN = \partial / \partial \bmn$ for the Neumann condition, 
and $\calN = \gamma_1 + \gamma_2 \cdot \partial / \partial \bmn~(\gamma_1, \gamma_2 \in \bbR)$ for the Robin condition. 
Hereafter we fix $\bmi \in \bbZ^\Dim$ and let
\begin{equation}
\calX = \calX(\bmi) = \calS_{\mathrm{PLG}}(\ibold) \cup \left\{ \calS_\bmi \right\}
\end{equation}
be the stencil for discretizing $\calL$ on $\calC_\bmi$,
where $\calS_{\mathrm{PLG}}(\ibold) = \left\{\calC_{\bmj_1}, \calC_{\bmj_2}, \cdots, \calC_{\bmj_N}\right\}$
 is the poised lattices generated by PLG algorithm with $N=\dim{\Pi_n^\Dim}$; 
 see Figure~\ref{fig:exampleStencil} for an example.
Traditional methods usually select larger stencil or keep adding nearby cells 
 to ensure the solvability of the linear system.
But blindly expanding the stencil may lead to a large number of redundant points,
 deteriorating computational efficiency.
This is precisely an advantage of our PLG algorithm.

\begin{figure}[htb]
	\centering
	\includegraphics[width=0.5\linewidth]{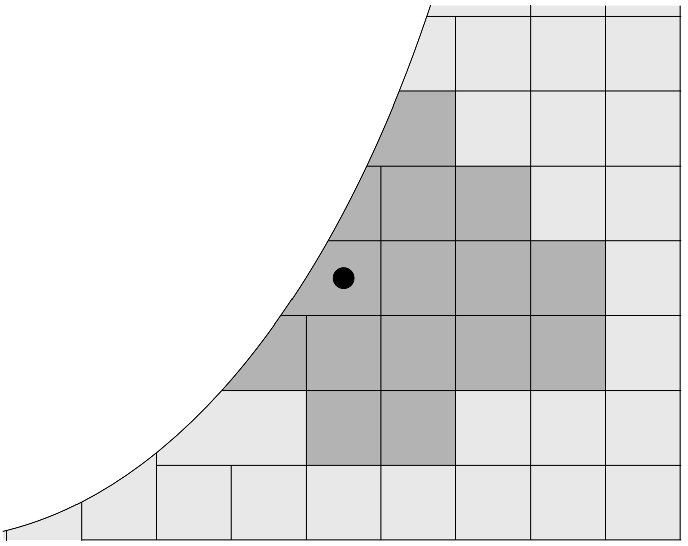}
	\caption{An example of poised lattices covered by dark-shaded cells 
	 in finite-volume discretization for $\Dim = 2$ and $n=4$.
	 The cell $\calC_{\bmi}$ is marked by a bullet $\bullet$.}
	\label{fig:exampleStencil} 
\end{figure}

We have assumed that $\calX$ contains the boundary face $\calS_\bmi$
and the treatment of the other case is similar. 
In the following, 
we reserve the symbol $n$ for the degree of polynomial fitting
(and of the lattice $\calX$, of course), 
the index $k\ (1 \le k \le N)$ for the interpolation sites, 
and the index $j\ (1 \le j \le N)$ for the basis of $\Pi_n^\Dim$.
Given the cell averages and the boundary data
\begin{equation}
\begin{aligned}
\hat{u} &= \left[ u_1, \cdots, u_N, u_{\mathrm{b}} \right]^T \\
&= \left[ \brk{u}_{\bmj_1}, \cdots , \brk{u}_{\bmj_N}, \bbrk{\calN u}_{\bmi} \right]^T
\in \bbR^{N+1}, 
\end{aligned}
\end{equation}
the goal is to determine the coefficients
\begin{equation}
\hat{\beta} = \left[ \beta_1, \cdots, \beta_N, \beta_{\mathrm{b}}\right]^T
\in \bbR^{N+1}, 
\end{equation}
such that the linear approximation formula
\begin{equation}
\begin{aligned}
\brk{\calL u}_\bmi &= \sum_{k=1}^N \beta_k u_k 
+ \beta_{\mathrm{b}} u_{\mathrm{b}} + \bigO(h^{n-q+1}) \\
&= \hat{\beta}^T \hat{u} + \bigO(h^{n-q+1})
\end{aligned}
\label{eq:goalOfApproximation}
\end{equation}
holds for all sufficiently smooth function $u : \bbR^\Dim \rightarrow \bbR$. 

We establish the equations on $\hat{\beta}$ 
by requiring that 
(\ref{eq:goalOfApproximation})
holds exactly for all $u \in \Pi_n^\Dim$, i.e.\ 
\begin{equation}
\brk{\calL u}_{\bmi} = \sum_{k=1}^N \beta_k u_k + \beta_{\mathrm{b}} u_{\mathrm{b}}
\end{equation}
for all $u \in \Pi_n^\Dim$. 
Since $\left\{ \phi_j \right\}_{j=1}^N$ is a basis of $\Pi_n^\Dim$, 
it is equivalent to say
\begin{equation}
\brk{\calL \phi_j}_{\bmi} = \sum_{k=1}^N \beta_k \brk{\phi_j}_{\bmj_k}
+ \beta_{\mathrm{b}} \bbrk{\phi_j}_{\bmi}
\end{equation}
for $j = 1, \cdots, N$, or simply
\begin{equation}
M \hat{\beta} = \hat{L}, 
\end{equation}
where 
\begin{align}
\label{eq:detailedSampleMatrix}
M &= 
\begin{bmatrix}
\brk{\phi_1}_{\jbold_1} & \brk{\phi_1}_{\jbold_2} & \cdots & \brk{\phi_1}_{\jbold_N} & \bbrk{\calN \phi_1}_{\bmi} \\
\brk{\phi_2}_{\jbold_1} & \brk{\phi_2}_{\jbold_2} & \cdots & \brk{\phi_2}_{\jbold_N} & \bbrk{\calN \phi_2}_{\bmi} \\
\vdots & \vdots & \ddots & \vdots & \vdots \\
\brk{\phi_N}_{\jbold_1} & \brk{\phi_N}_{\jbold_2} & \cdots & \brk{\phi_N}_{\jbold_N} & \bbrk{\calN \phi_N}_{\bmi} \\
\end{bmatrix}
\in \bbR^{N \times (N+1)}, 
%
\end{align}
and
\begin{align}
\hat{L} &= 
\begin{bmatrix}
\brk{\calL \phi_1}_{\bmi} & \brk{\calL \phi_2}_{\bmi} & \cdots & \brk{\calL \phi_N}_{\bmi}
\end{bmatrix}^T\in \bbR^N. 
\end{align}
This is an underdetermined system on $\hat{\beta}$
and we seek the weighted minimum norm solution via
\begin{equation}
\min_{\beta \in \bbR^{N+1}} \left\Vert \beta \right\Vert_{W^{-1}} \quad
\mathrm{s.t.} \quad M \beta = \hat{L}. 
\end{equation}
Since $M$ has full row rank,
it follows from Proposition~\ref{prop:leastSquares} that
\begin{equation}
\hat{\beta} = W M^T \left( M W M^T \right)^{-1} \hat{L}. 
\label{eq:derivedCoefAgain}
\end{equation}
The weight matrix $W$ is used to penalize large coefficients 
on the cells far from $\calC_{\bmi}$.
For simplicity, we set $W^{-1} = \diag\left( w_1, \cdots, w_N, w_{\mathrm{b}} \right)$,
where 
\begin{subequations}
\begin{align}
w_k &= \max \left\{ \Vert \bmj_k - \bmi \Vert_2, w_{\mathrm{min}} \right\}, \\ 
w_{\mathrm{b}} &= w_{\mathrm{min}}, 
\end{align}
\label{eq:weightDef}
\end{subequations}
and $w_{\mathrm{min}}$ is set to $1/2$ to avoid zero weights.

In practice, to maintain a good conditioning, 
 we use the following basis (with recentering and scaling)
\begin{equation}
\Phi_n^D(h;\bmp) = \left\{ \left( \frac {\bmx - \bmp} {h} \right)^{\bmalpha} : 
0 \le \vert \bmalpha \vert \le n \right\}, 
\end{equation}
where $\bmp \in \bbR^D$ is the center (or simply the center of the bounding box) 
of the interpolation sites.

\subsection{Discrete elliptic problems}
\label{sec:discreteElliptic}

In this section, we discuss the fourth-order discretization
of the constant-coefficient elliptic equation (\ref{eq:ccEllipticEq})
with the boundary condition
\begin{equation}
\bmn \cdot \nabla u = g \quad \text{on } \partial\Omega.
\end{equation}
For ease of exposition, we specify the Neumann condition
but the treatment of the other cases is similar. 
Define $\hat{u}$ to be the vector consisting of 
the cell averages of the unknown function $u$, 
and define $\hat{f}$ accordingly. 
Define $\hat{g}$ to be the vector consisting of 
the boundary-face averages of the boundary condition $g$. 
Combining the fourth-order difference formula (\ref{eq:stdLap})
and the third-order 
finite-volume PLG approximation (\ref{eq:goalOfApproximation}), 
we obtain the discretization
\begin{equation}
L_0\left(\hat{u}, \hat{g}\right) = \hat{f}. 
\label{eq:PoissonDiscretization}
\end{equation}
Note that the linearity of $L_0(\cdot, \cdot)$ in each variable is implied by our construction;
hence we can rewrite (\ref{eq:PoissonDiscretization}) into
\begin{equation}
L \hat{u} + N \hat{g} = \hat{f}, 
\label{eq:PoissonLinearDiscretization}
\end{equation}
where $L$ and $N$ are both matrices. 
Next we transform (\ref{eq:PoissonLinearDiscretization}) into the residual form
\begin{equation}
L \hat{u} = \hat{r} \coloneqq \hat{f} - N \hat{g}, 
\label{eq:PoissonResidualForm}
\end{equation}
and further split it into two row blocks, i.e.\ 
\begin{equation}
\begin{bmatrix}
L_{11} & L_{12} \\
L_{21} & L_{22}
\end{bmatrix}
\begin{bmatrix}
\hat{u}_1 \\
\hat{u}_2
\end{bmatrix}
=
\begin{bmatrix}
\hat{r}_1 \\
\hat{r}_2
\end{bmatrix}
. 
\label{eq:PoissonSplitDiscretization}
\end{equation}
The split $\hat{u} = \left[ \hat{u}_1^T \ \hat{u}_2^T \right]^T$
is based on the discretization type:
if the cell $\ibold$ is a regular cell, 
then the cell average $\brk{u}_{\bmi}$ is contained in $\hat{u}_1$; 
otherwise it is contained in $\hat{u}_2$. 
Consequently, the sub-block $L_{11}$ has a regular structure\footnote{
If Poisson's equation is considered, the sub-block $L_{11}$ resembles
the fourth-order, nine-point discretization
of Poisson's equation in 2D rectangular domains; see e.g.~\cite{Zhang2014}. }
whereas the structures of the other sub-blocks are only known to be sparse. 

\section{Multigrid algorithm}
\label{sec:multigrid}

In this section, we present the multigrid algorithm
for solving the discrete elliptic equation (\ref{eq:PoissonSplitDiscretization}). 
The overall procedure is geometric-based. 
First we initialize a hierarchy of successively coarsened grids
\begin{equation}
\Omega^{\ast} = \left\{ \Omega^{(m)} : 0 \le m \le M \right\}, 
\end{equation}
where $\Omega^{(m)}$ is obtained by embedding $\Omega$
into a rectangular grid of step size $h^{(m)} = 2^m h^{(0)}$. 
On each level 
we construct the discrete elliptic operator $L^{(m)}$ 
via the finite-volume PLG discretization discussed in the last section. 
We stop at some $\Omega^{(M)}$ where the direct solution (say, via LU factorization)
of the discrete linear system is inexpensive. 

We should mention that the finite-volume PLG discretization
prohibits the direct application of the traditional geometric multigrid method.
This is because:
(1) The Gauss-Siedel and the (weighted) Jacobi iteration 
are not guaranteed to converge due to the presence of the irregular sub-blocks
$L_{12}, L_{21}$ and $L_{22}$ in (\ref{eq:PoissonSplitDiscretization}); 
(2) Simple grid-transfer operators, 
such as the volume-weighted restriction and the linear interpolation
cannot be applied near the irregular boundary. 
In the following, we adhere to the multigrid framework \cite{Briggs:A_Multigrid_Tutorial}
but propose modifications to the V-cycle (Algorithm~\ref{alg:VCycle}) components
so that the algorithm is still effective with high-order scheme 
 and arbitrarily complex irregular boundaries.
To improve V-cycles,
 the full multigrid (FMG) cycle, 
 which utilizes the multigrid concept to build an accurate initial guess, 
 is also implemented; see Figure~\ref{fig:FMG-cycle}.
Although FMG cycles are more expensive than V-cycles,
 it has been widely used because of the better initial guesses 
 and the optimal complexity~\cite{Briggs:A_Multigrid_Tutorial}.

\begin{algorithm}[htbp]
\caption{\textbf{VCycle}$(L^{(m)}, \hat{u}^{(m)}, \hat{r}^{(m)}, \nu_1, \nu_2)$}
\label{alg:VCycle}
\begin{algorithmic}[1]
\Require
An integer $m$ indicating the level, with 0 being the finest; \\
the discrete elliptic operator $L^{(m)}$ on the $m$th level; \\
the initial guess $\hat{u}^{(m)}$; \\
the residual $\hat{r}^{(m)}$ on the $m$th level; \\
the smooth parameters $\nu_1, \nu_2$. 
\Ensure The solution to the algebraic equation
$L^{(m)} \hat{u}^{(m)} = \hat{r}^{(m)}$. 
\If{$m = M$}
\State Solve $L^{(M)} \hat{u}^{(M)} = \hat{r}^{(M)}$ using the bottom solver.
\label{line:bottomSolver}
\Else
\State Apply the smoother to $L^{(m)} \hat{u}^{(m)} = \hat{r}^{(m)}$ for $\nu_1$ times.
\label{line:preSmooth}
\State Compute the coarse residual by restriction: \\
\quad $\hat{r}^{(m+1)} = \text{\textbf{Restrict}}(\hat{r}^{(m)} - L^{(m)} \hat{u}^{(m)}).$
\label{line:restrict}
\State Recursively solve the algebraic equation on the coarse grid: \\
\quad $\hat{u}^{(m+1)} = \mathbf{0}$, \\
\quad $\hat{u}^{(m+1)} = \text{\textbf{VCycle}}(L^{(m+1)}, \hat{u}^{(m+1)}, \hat{r}^{(m+1)}, \nu_1, \nu_2)$. 
\State Prolong the correction and update the solution: \\
\quad $\hat{u}^{(m)} = \hat{u}^{(m)} + \text{\textbf{Prolong}}(\hat{u}^{(m+1)})$. 
\label{line:prolong}
\State Apply the smoother to $L^{(m)} \hat{u}^{(m)} = \hat{r}^{(m)}$
for $\nu_2$ times. 
\label{line:postSmooth}
\EndIf \\
\Return $\hat{u}^{(m)}$.
\end{algorithmic}
\end{algorithm}

\begin{figure}[htbp]
	\center
	\includegraphics[width=0.8\linewidth]{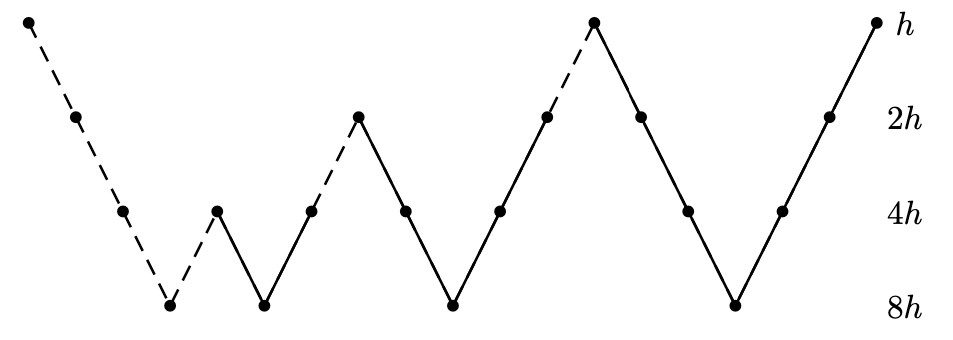}
	\caption{Illustration of the FMG cycle for a grid with four levels.
	FMG begins with a descent to the coarsest grid $\Omega^{8h}$;
	this is represented by the downward dashed line segments. 
	Then the solution is prolongated to $\Omega^{4h}$ 
	 and used as the initial guess to the V-cycle on $\Omega^{4h}$. 
	This ``prolongation $+$ V-cycle" process is repeated recursively: 
	 the prolongation is represented by an upward dashed line 
	 and the V-cycles are represented by the solid lines.}
	\label{fig:FMG-cycle}
\end{figure}

\subsection{Multigrid components}

First we discuss the smoother used in Algorithm~\ref{alg:VCycle}.
To avoid cluttering,
we temporarily omit the superscript $^{(m)}$ indicating the grid level, 
and use the prime variables to denote the quantities after one iteration of the smoother. 
Then, to ensure convergence,
we have made the following adjustments.
\begin{enumerate}[label={(SMO-\arabic*)}, noitemsep, leftmargin = *]
\item The $\omega$-weighted Jacobi iteration 
\begin{equation}
\hat{u}'_1 = D^{-1} \left[ (1-\omega) D + \omega O \right] \hat{u}_1 
+ \omega D^{-1} O \left( \hat{r}_1 - L_{12} \hat{u}_2 \right)
\end{equation}
is applied in a pointwise fashion to $\hat{u}_1$, 
where $D$ is the diagonal part of $L_{11}$
and $O = D - L_{11}$ is the off-diagonal part; 
\label{enum:SMO1}
\item A block update 
\begin{equation}
L_{22}\hat{u}'_2 = \hat{r}_2 - L_{21} \hat{u}'_1
\label{eq:updateByBlock}
\end{equation}
is applied to obtain $\hat{u}'_2$. 
\label{enum:SMO2}
\end{enumerate}

Several comments are in order. 
First, 
it is favorable to pre-compute the LU factorization of the sub-block $L_{22}$,
allowing for solving the linear system (\ref{eq:updateByBlock}) with 
multiple right-hand-sides with minimal cost. 
Second, we have used the updated value $\hat{u}_1'$ 
instead of $\hat{u}_1$ in (\ref{eq:updateByBlock}). 
Hence, the overall iteration is more like Gauss-Siedel in a blockwise fashion.
After two iterations of the smoother, we have
\begin{subequations}
\begin{align}
\label{eq:errorOfE1}
\hat{e}'_1 &= D^{-1} \left[ (1-\omega) D + \omega O \right] \hat{e}_1, \\
\label{eq:errorOfE2}
\hat{e}'_2 &= \mathbf{0},
\end{align}
\label{eq:errorOfSmoother}
\end{subequations}
where $\hat{e} = [\hat{e}_1^T ~ \hat{e}_2^T ]^T$ and its prime version
are the solution errors in (\ref{eq:PoissonSplitDiscretization})
before and after the iteration, respectively.
From (\ref{eq:errorOfE1}) we see that
the error on the interior cells is essentially damped by an $\omega$-weighted Jacobi iteration,
exactly what we would expect from an effective smoother. 

By virtue of (\ref{eq:errorOfE2}), 
not only the error but also the residual 
corresponding to the $\hat{u}_2$ part are identically zero
after two or more iterations of the smoother. 
Consequently, 
grid-transfer operators 
that carefully handles the irregular cells near the domain boundary
are not mandatory. 
In practice, 
we apply the volume weighted restriction
\begin{equation}
\brk{u}^{(m+1)}_{\lfloor \bmi / 2 \rfloor} = 
2^{-\Dim} \sum_{\bmj \in \left\{ 0, 1 \right\}^\Dim} \brk{u}^{(m)}_{\bmi + \bmj}
\end{equation}
and the patchwise-constant interpolation
\begin{equation}
\brk{u}^{(m)}_{\bmi} = \brk{u}^{(m+1)}_{\lfloor \bmi / 2 \rfloor}
\end{equation}
for regular cells, 
while leaving the correction and the residual for partial cells to zero. 

On the coarsest level
we invoke the LU solver on $L^{(M)} \hat{u}^{(M)} = \hat{r}^{(M)}$
where the LU factorization of $L^{(M)}$ is computed beforehand. 

\subsection{Optimal complexity of LU factorization for $L_{22}$}
\label{sec:order_L22}

Since we frequently need to solve the linear system (\ref{eq:updateByBlock}), 
it is essential to find efficient methods.
Traditional LU factorization has a complexity of $\bigO(N^3)$, 
where $N$ is the number of unknowns. 
In this section, 
we present an optimal LU factorization for $L_{22}$ 
achieved by reordering the unknowns, 
which has a complexity of only $\bigO(N)$.
To be specific, 
let $\calI_{\mathrm{bdry}}$ be the collection of all interface cells, i.e.
\[
\calI_{\mathrm{bdry}} \coloneqq \left\{ \bmi \in \bbZ^2 : \calC_{\bmi} \neq \emptyset, \rmC_{\bmi} \right\}. 
\]
Let $\gamma : [0, 1] \to \mathbb{R}^2$ 
 be a parameterization of the Jordan curve $\partial\Omega$ with $\gamma(0) = \gamma(1)$.
For each $\ibold\in\mathcal{I}_{\mathrm{bdry}}$, 
define $s(\ibold)\in [0,1)$ such that
\begin{equation}
\mathrm{dist}(\partial\Omega,R_\ibold) = \Vert R_\ibold - \gamma(s(\ibold)) \Vert,
\end{equation}
where $R_\ibold$ is the cell center of $\mathbf{C}_\ibold$. 
Next we define the linear order $``\leq"$ on $\mathcal{I}_\mathrm{bdry}$
by requiring
\begin{equation}
\label{eq:linearOrder}
\ibold \leq \jbold \quad \mathrm{iff} \quad s(\ibold)\leq s(\jbold).
\end{equation}
We compute the ordering for the irregular cells 
according to Algorithm \ref{alg:order}; 
see Figure \ref{fig:order} for an illustration.
To measure the effectiveness of this ordering, 
we consider a decomposition of the reordered matrix $L_{22}$, 
\begin{equation}
\label{eq:L_22_decomposition}
L_{22} = 
L_{22,c}
+ 
\begin{bmatrix}
O & L_{22,u} \\
O & O
\end{bmatrix} 
+ 
\begin{bmatrix}
O & O \\
L_{22,l} & O
\end{bmatrix},
\end{equation}
where $L_{22,c}$ is a banded matrix,
and $L_{22,u}, L_{22,l}$ are square sub-matrices with small dimensions. 
The \emph{cyclic bandwidth} of $L_{22}$,
defined as the number $\lambda(L_{22,c}, L_{22,u}, L_{22,l})$, 
which is the maximum of the bandwidth of $L_{22,c}$
and the dimensions of $L_{22,u}$ and $L_{22,l}$, 
is associated with each decomposition in the form of 
(\ref{eq:L_22_decomposition}). 
In practical numerical examples, 
the cyclic bandwidth of $L_{22}$ is found to be independent of the grid size 
and is no larger than $20$,
even for very complex geometries. 
Consequently, the complexity of the LU factorization of the matrix $L_{22}$ 
can be kept linear with its dimension (see Section \ref{sec:efficiency}), 
provided no pivoting is performed. 
For details on the LU factorization of the banded matrix, see~\cite{Golub}.

\begin{algorithm}
	\caption{An order over the irregular cells.}  
	\label{alg:order}  
	\begin{algorithmic}[1]
	\Require { The collection of all interface cells $\mathcal{I}_{\mathrm{bdry}}$ 
	ordered by $``\leq"$ defined in \eqref{eq:linearOrder}.}
	\Ensure{ 
	An order over the irregular cells, represented by
	a non-negative integer $P(\ibold)$ for each
	irregular cell $\ibold$}.

	\State Initialize $P(\ibold) = -1$ for each irregular cell $\ibold$
	
	\State $k \leftarrow 0$
	
	\For{$\ibold\in (\mathcal{I}_{\mathrm{bdry}}, \leq)$}
		\For{ all $\jbold\in \bbZ^2$ with $\left\Vert \ibold-\jbold \right\Vert_{\infty} \le 2$}
			\If{ cell $\jbold$ is irregular $\mathrm{and}$ $P(\jbold)=-1$}
				\State $P(\jbold) = k$
				\State $k\leftarrow k+1$
			\EndIf
		\EndFor
	\EndFor
	\end{algorithmic}
\end{algorithm}

\begin{figure}[htbp]
\centering
\includegraphics[width=0.5\textwidth]{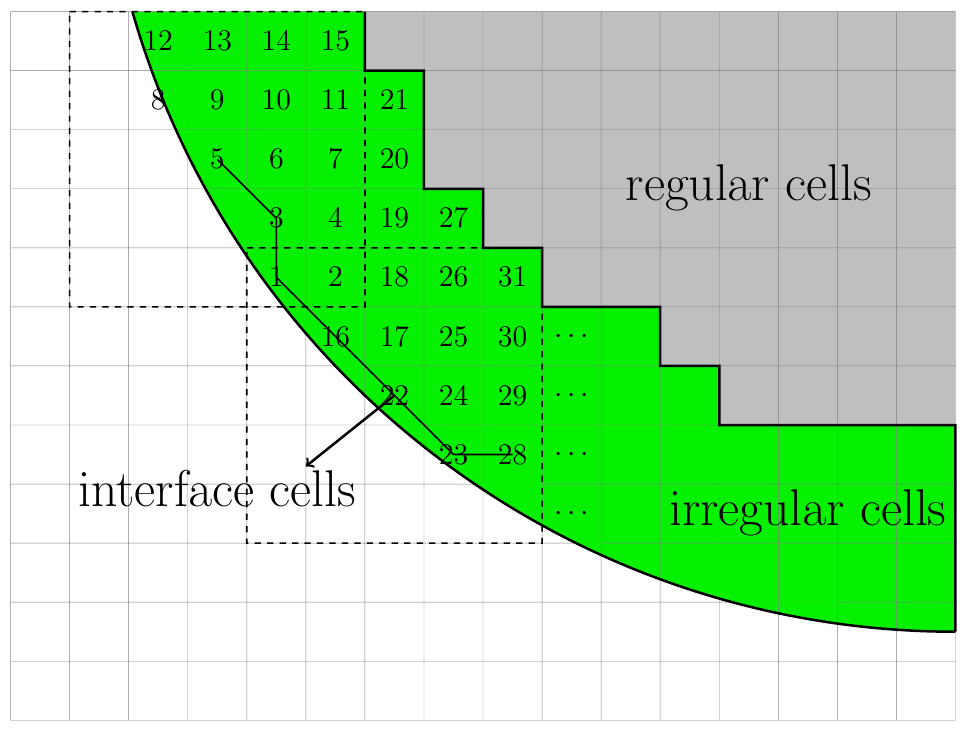}
\caption{
An illustration for the linear order output by Algorithm \ref{alg:order}.
The regular cells are shaded light gray, while the irregular cells are shaded yellow. 
}
\label{fig:order} 
\end{figure}



\subsection{Complexity}
\label{sec:complexity}

Let $h = h^{(0)}$ be the step size of the finest grid. 
We assume the sub-vector $\hat{u}_1$ in (\ref{eq:PoissonSplitDiscretization})
has dimension $\bigO(h^{-2})$
and the sub-vector $\hat{u}_2$ has dimension $\bigO(h^{-1})$. 

\begin{enumerate}[label={(\alph*)}, leftmargin=*]
\item \textbf{Setup}. 
Consider the setup time on the finest level.
During the construction of the discrete Laplacian and cross operators, 
the finite-volume PLG discretization is invoked to obtain the $\left[ L_{21} \ L_{22} \right]$ blocks. 
This requires $\bigO(h^{-1})$ time since the PLG discretization is $\bigO(1)$ for each cell
and these blocks are sparse with $\bigO(h^{-1})$ rows. 
{
We then compute the sparse LU factorization of $L_{22}$, 
which takes $\bigO(h^{-1})$ time thanks to its sparsity and small bandwidth; 
see also the numerical evidences in Section \ref{sec:efficiency}. 
}
For the complete hierarchy of grids, 
the complexity is bounded by a constant multiple of that on the finest level, 
so the total setup time is $\bigO(h^{-1})$.

\item \textbf{Solution}. 
We first bound the complexity of a single V-cycle (Algorithm~\ref{alg:VCycle}) 
and a single FMG cycle (Figure~\ref{fig:FMG-cycle}).
The bottom solver requires $\bigO(1)$ time since it operates on the coarsest level only.
On the finest level, 
the pointwise Jacobi iteration on $\hat{u}_1$ ~\ref{enum:SMO1} requires $\bigO(h^{-2})$ time.
The blockwise update on $\hat{u}_2$~\ref{enum:SMO2}
requires $\bigO(h^{-1})$ time, 
provided the LU factorization of $L_{22}$ is available. 
Each of the restriction and the prolongation requires $\bigO(h^{-2})$. 
Since the complexity of the complete hierarchy
is bounded by a constant multiple of that on the finest level, 
we know that both a single V-cycle and a single FMG cycle 
have a complexity of $\bigO(h^{-2})$. 

To bound the complexity of the overall solution procedure,
we still need to estimate the number of V-cycles and FMG cycles 
needed to reduce the initial residual to a given fraction. 
Considering the recursive nature of V-cycle
and the smoothing property (\ref{eq:errorOfSmoother}),
we expect that error modes of all frequency 
can be effectively damped by the multigrid algorithm,
similar to their behavior on rectangular domains, 
with V-cycles having a convergence factor 
$\gamma$ that is independent of $h$.
Since V-cycles must reduce the algebraic error from $\bigO(1)$ to $\bigO(h^4)$,
the number $\mu$ of V-cycles required must satisfy 
$\gamma^\mu = \bigO(h^4)$ or $\mu = \bigO(h^{-1})$.
Because the cost of a single V-cycle is $\bigO(h^{-2})$,
the cost of V-cycles is $\bigO(h^{-2}\log (h^{-1}))$.
The FMG cycle costs a little more per cycle than the V-cycle. 
However, a properly designed FMG cycle can be much more effective overall 
because it supplies a very good initial guess to the final V-cycles on the finest level.
This result in only $\bigO(1)$ V-cycles are needed on the finest grid
(see~\cite{Brandt:Multigrid_Techniques} for a more detailed discussion).
This means that the total complexity of FMG cycles is $\bigO(h^{-2})$, which is optimal.
This is also confirmed by numerical examples in Section~\ref{sec:efficiency}.
\end{enumerate}

We point out that the solution time of $\bigO(h^{-2})$ is the best possible bound 
we can achieve in irregular domains, 
since solving an elliptic equation in a rectangular grid with $\bigO(h^{-2})$ cells
using the geometric multigrid method has exactly $\bigO(h^{-2})$ complexity.
Detailed numerical results on the efficiency of the multigrid algorithm
are reported in Section~\ref{sec:efficiency}.

\section{Numerical tests}
\label{sec:Tests}

In this section we demonstrate the accuracy and efficiency of our method
by various test problems in irregular domains. 
Define the $L^p$ norms
\begin{equation}
\Vert u \Vert_p = \left\{
\begin{aligned}
& \left( \frac {1} {\Vert \Omega \Vert} \sum \left\Vert \calC_\bmi \right\Vert \cdot
\left\vert \brk{u}_{\bmi} \right\vert^p \right)^{1/p}  \quad &&p = 1, 2, \\
& \max \ \left\vert \brk{u}_{\bmi} \right\vert \quad &&p = \infty,
\end{aligned}
\right.
\end{equation}
where the summation and the maximum are taken over
the non-empty cells in the computational domain.

\subsection{Convergence tests}
\label{sec:convergence}

\textbf{Problem 1.} 
Consider a problem~\cite[Problem 3]{Johansen1998} of Poisson's equation
on the irregular domain $\Omega = R \cap \Omega_1$, 
where $R$ is the unit box centered at the origin and 
\begin{equation}
\Omega_1 = \left\{ \left( r, \theta \right) : r \ge 0.25 + 0.05 \cos 6 \theta \right\}. 
\end{equation}
Here $\left( r, \theta \right)$ are the polar coordinates satisfying
$(x_1, x_2) = (r \cos \theta, r \sin \theta)$.
A Dirichlet condition is imposed on the rectilinear sides $\partial R$
while a Neumann condition is imposed on the irregular boundary $\partial \Omega_1$. 
Both the right-hand-side and the boundary conditions
are derived from the exact solution
\begin{equation}
u(r, \theta) = r^4 \cos 3 \theta. 
\end{equation}

The numerical solution and the solution error on the grid of $h=1/80$
are shown in Figure~\ref{fig:SolOfProblem1} and~\ref{fig:ErrOfProblem1}, respectively.
The error appears oscillating near the irregular boundary
due to the rapidly varying truncation errors 
induced by the PLG discretization. 
However, it does not affect the fourth-order convergence of the numerical solution 
as confirmed by the grid-refinement tests. 

The truncation errors and the solution errors
are listed in Table~\ref{tab:truncationErrorsOfProblem1} and~\ref{tab:solutionErrorsOfProblem1}.
The convergence rates of the truncation errors 
are asymptotically close to 3, 4 and 3.5
in $L^\infty$ norm, $L^1$ norm and $L^2$ norm, respectively. 
This suggests that the PLG approximation is third-order accurate
on the boundary cells of codimension one 
and fourth-order accurate on the interior cells. 
The convergence rates of the solution errors 
are close to 4 in all norms as expected. 
Also included in the tables
are the comparisons with the second-order EB method
by Johansen and Colella~\cite{Johansen1998}.
Our method is much more accurate than the second-order EB method
in terms of solution errors. 
In particular, the $L^\infty$ norm of the error on the grid of $h=1/80$ in our method 
is smaller by a factor of 40 than that on the grid of $h=1/320$ in the EB method.
On the grid of $h=1/320$, the $L^\infty$ norm of the error in our method 
is about $1/8000$ of that in the second-order EB method. 

\begin{table}[htb]
\centering
\begin{tabular}{cccccccc}
\hline
\multicolumn{8}{c}{Truncation errors of the EB method by Johansen and Colella~\cite{Johansen1998}} \\
\hline
       $h$ &       1/40 &       rate &       1/80 &       rate &      1/160 &       rate &      1/320 \\
\hline
$L^\infty$ &   1.66e$-$03 &       2.0 &   4.15e$-$04 &       2.0 &   1.04e$-$04 &       2.0 &   2.59e$-$05 \\
\hline
\multicolumn{8}{c}{Truncation errors of the current method} \\
\hline
       $h$ &       1/40 &       rate &       1/80 &       rate &      1/160 &       rate &      1/320 \\
\hline       
$L^\infty$ &   2.94e-04 &       0.78 &   1.71e-04 &       2.83 &   2.41e-05 &       2.95 &   3.13e-06 \\
     $L^1$ &   1.03e-05 &       3.74 &   7.70e-07 &       4.16 &   4.30e-08 &       3.87 &   2.94e-09 \\
     $L^2$ &   3.30e-05 &       3.00 &   4.13e-06 &       3.65 &   3.29e-07 &       3.45 &   3.01e-08 \\
\hline     
\end{tabular}
\caption{Truncation errors of Problem 1, 
with Dirichlet condition on the rectilinear sides
and Neumann condition on the irregular boundary.
Comparisons with a second-order EB method~\cite{Johansen1998} are shown.
}
\label{tab:truncationErrorsOfProblem1}
\end{table}%

\begin{table}[htb]
\centering
\begin{tabular}{cccccccc}
\hline
\multicolumn{8}{c}{Solution errors of the EB method by Johansen and Colella~\cite{Johansen1998}} \\
\hline
       $h$ &       1/40 &       rate &       1/80 &       rate &      1/160 &       rate &      1/320 \\
\hline       
$L^\infty$ &   4.78e$-$05 &       1.85 &   1.33e$-$05 &       1.98 &   3.37e$-$06 &       1.95 &   8.72e$-$07 \\
\hline
\multicolumn{8}{c}{Solution errors of the current method} \\
\hline
       $h$ &       1/40 &       rate &       1/80 &       rate &      1/160 &       rate &      1/320 \\
\hline       
$L^\infty$ &   3.68e-07 &       4.08 &   2.17e-08 &       3.70 &   1.67e-09 &       3.92 &   1.10e-10 \\
     $L^1$ &   3.77e-08 &       3.85 &   2.62e-09 &       4.41 &   1.23e-10 &       3.79 &   8.86e-12 \\
     $L^2$ &   6.24e-08 &       4.12 &   3.59e-09 &       4.19 &   1.97e-10 &       3.81 &   1.41e-11 \\
\hline     
\end{tabular}
\caption{Solution errors of Problem 1, 
with Dirichlet condition on the rectilinear sides
and Neumann condition on the irregular boundary.
Comparisons with a second-order EB method~\cite{Johansen1998} are shown.
}
\label{tab:solutionErrorsOfProblem1}
\end{table}%

\begin{figure}[htbp]
    \centering
    \begin{subfigure}[b]{.6\textwidth}
        \centering
        \includegraphics[width=1.0\linewidth]{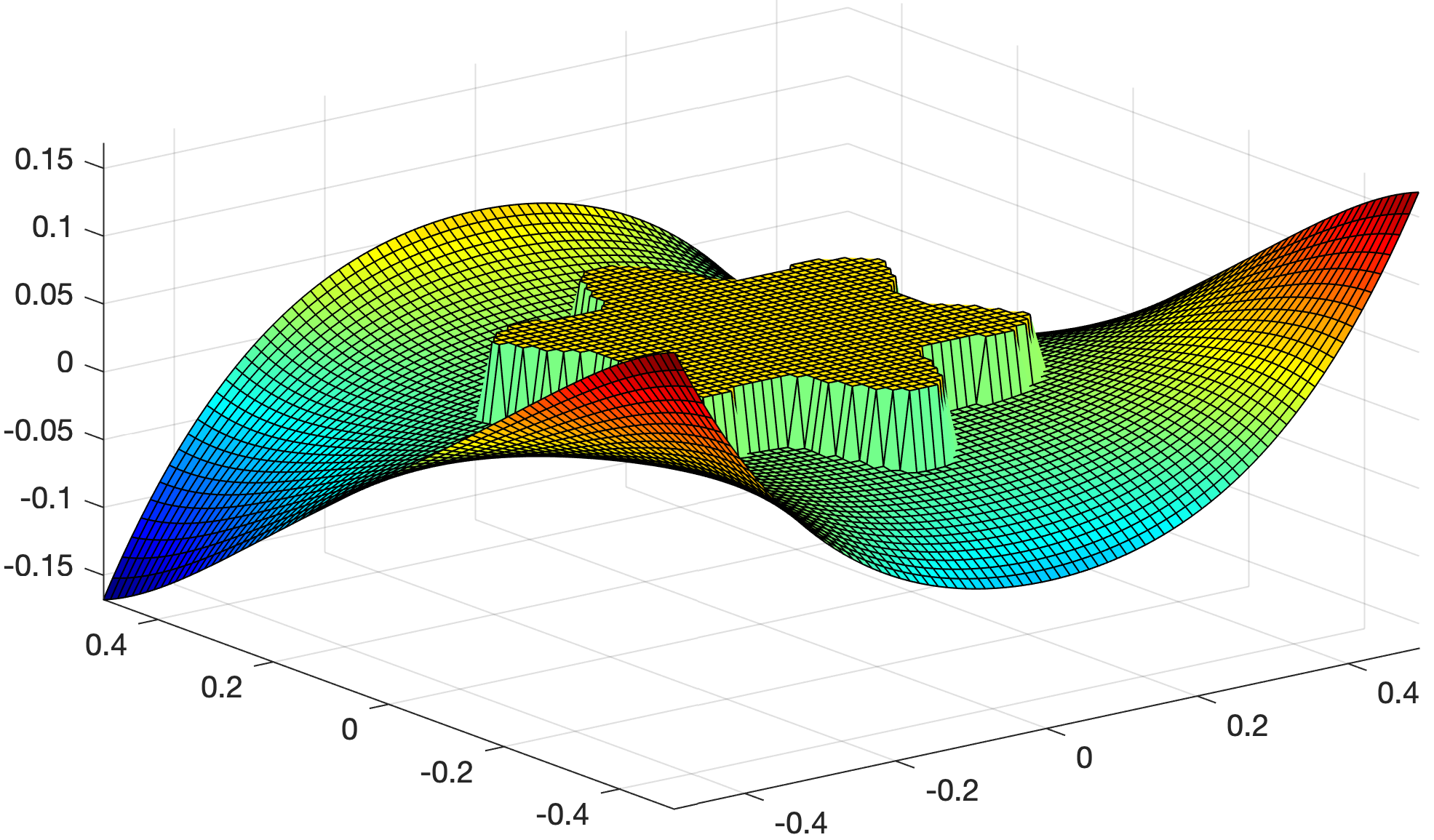}
        \caption{Numerical solution}
        \label{fig:SolOfProblem1}
    \end{subfigure}
    \begin{subfigure}[b]{.6\textwidth}
        \centering
        \includegraphics[width=1.0\linewidth]{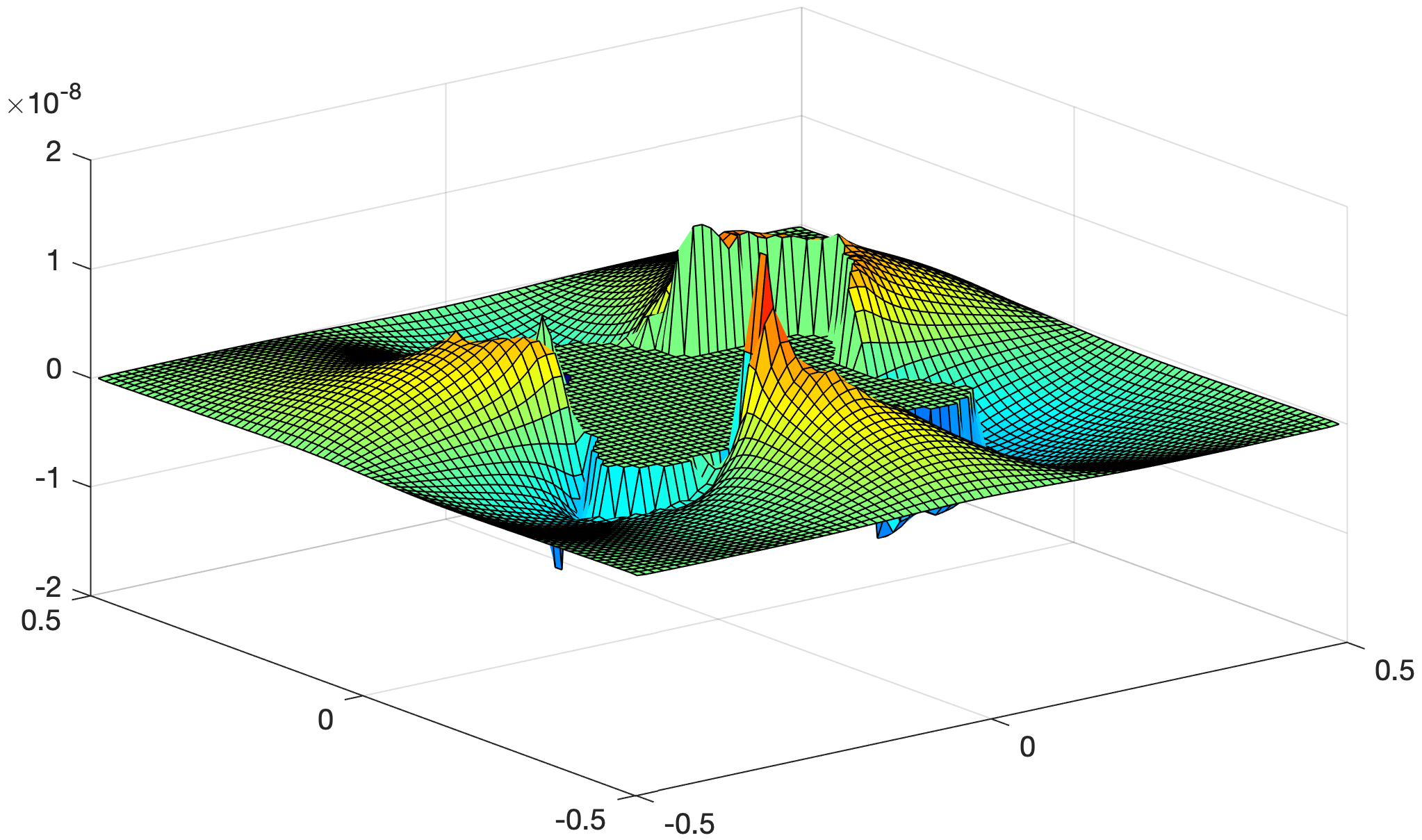}
        \caption{Solution error}
        \label{fig:ErrOfProblem1}
    \end{subfigure}
    \caption{Solving Problem 1 on the grid of $h=1/80$. }
\end{figure}

\textbf{Problem 2.} 
Consider a problem from~\cite[\S 5.1]{Devendran2017}
where Poisson's equation is solved on the domain $\Omega = R \cap \Omega_2$, 
$R$ is the unit box $[0, 1]^2$
and $\Omega_2$ is the exterior of the ellipse defined by
\begin{equation}
\sum_{i=1}^2 \left( \frac {x_i - c_i} {a_i} \right)^2 = 1. 
\end{equation}
The parameters are $(c_1, c_2) = (1/2, 1/2)$ and $(a_1, a_2) = (1/8, 1/4)$. 
A Dirichlet condition is imposed on the rectilinear sides $\partial R$, 
and either a Dirichlet or a Neumann condition 
is imposed on the irregular boundary $\partial \Omega_2$. 
The right-hand-side and the boundary conditions
are derived from the exact solution
\begin{equation}
u = \Pi_{i=1}^2 \sin(\pi x_i). 
\end{equation}

The solution errors with Neumann condition and Dirichlet condition 
on the grid of $h=1/256$
are shown in Figure~\ref{fig:errorNeumannOfProblem2}
and~\ref{fig:errorDirichletOfProblem2}, respectively.
The errors of both cases appear smooth over the entire domain, 
and approach zero on the Dirichlet boundaries. 
In the Neumann case, 
the maximum of error appears on the ellipse boundary;
while in the Dirichlet case, 
the maximum of error appears in the interior of the domain.

The solution errors 
are listed in Table~\ref{tab:errorsOfEllipseDN} and~\ref{tab:errorsOfEllipseDD}.
With either boundary condition, 
the solutions demonstrate a fourth-order convergence in all norms in the grid-refinement tests.
Also included in the tables are the comparisons with 
a recently developed fourth-order EB method~\cite{Devendran2017}.
It is observed that the solution errors of our method 
are smaller than those in~\cite{Devendran2017} in all norms.
For example, on the grid of $h=1/512$,
the solution error in our method 
is about one half of that in~\cite{Devendran2017} in the case of Neumann condition,
and about one third in the case of Dirichlet condition. 

\begin{table}[htb]
\centering
\begin{tabular}{cccccccc}
\hline
\multicolumn{8}{c}{A fourth-order EB method by Devendran et al.~\cite{Devendran2017}} \\
\hline
       $h$ &       1/64 &       rate &      1/128 &       rate &      1/256 &       rate &      1/512 \\
\hline       
$L^\infty$ &   1.83e-07 &      3.96 &   1.18e-08 &       3.98 &   7.43e-10 &      3.88 &   5.05e-11 \\
     $L^1$ &   6.90e-08 &       3.95 &   4.46e-09 &       3.97 &   2.83e-10 &      3.86 &   1.94e-11 \\
     $L^2$ &   8.67e-08 &       3.96 &   5.56e-09 &       3.98 &   3.51e-10 &      3.87 &   2.40e-11 \\
\hline
\multicolumn{8}{c}{Current method} \\
\hline
       $h$ &       1/64 &       rate &      1/128 &       rate &      1/256 &       rate &      1/512 \\
\hline       
$L^\infty$ &   9.17e-08 &       3.87 &   6.29e-09 &       3.87 &   4.31e-10 &       4.04 &   2.62e-11 \\
     $L^1$ &   1.99e-08 &       3.49 &   1.76e-09 &       3.67 &   1.38e-10 &       4.12 &   7.99e-12 \\
     $L^2$ &   2.66e-08 &       3.55 &   2.27e-09 &       3.72 &   1.73e-10 &       4.12 &   9.96e-12 \\
\hline     
\end{tabular}
\caption{Solution errors of Problem 2, 
with Dirichlet condition on the rectilinear sides
and Neumann condition on the irregular boundary.
Comparisons with a fourth-order EB method~\cite{Devendran2017} are shown. }
\label{tab:errorsOfEllipseDN}
\end{table}%

\begin{table}[htb]
\centering
\begin{tabular}{cccccccc}
\hline
\multicolumn{8}{c}{A fourth-order EB method by Devendran et al.~\cite{Devendran2017}} \\
\hline
       $h$ &       1/64 &       rate &      1/128 &       rate &      1/256 &       rate &      1/512 \\
\hline       
$L^\infty$ &   4.15e-08 &      4.04 &   2.53e-09 &       4.01 &   1.56e-10 &      3.84 &   1.09e-11 \\
     $L^1$ &   1.91e-08 &       4.02 &   1.17e-09 &       3.99 &   7.37e-11 &      3.82 &   5.20e-11 \\
     $L^2$ &   2.30e-08 &       4.04 &   1.40e-09 &       4.01 &   8.72e-11 &      3.83 &   6.11e-12 \\
\hline
\multicolumn{8}{c}{Current method} \\
\hline
       $h$ &       1/64 &       rate &      1/128 &       rate &      1/256 &       rate &      1/512 \\
\hline       
$L^\infty$ &   4.96e-08 &       5.79 &   8.95e-10 &       4.12 &   5.14e-11 &       3.94 &   3.35e-12 \\
     $L^1$ &   8.88e-09 &       4.34 &   4.39e-10 &       4.11 &   2.55e-11 &       4.16 &   1.42e-12 \\
     $L^2$ &   1.03e-08 &       4.35 &   5.06e-10 &       4.10 &   2.95e-11 &       4.13 &   1.68e-12 \\
\hline     
\end{tabular}
\caption{Solution errors of Problem 2, 
with Dirichlet condition on the domain boundaries.
Comparisons with a fourth-order EB method~\cite{Devendran2017} are shown. }
\label{tab:errorsOfEllipseDD}
\end{table}%

\begin{figure}[htbp]
\centering
\begin{subfigure}[b]{.45\linewidth}
  \centering
  \includegraphics[height=1.0\linewidth]{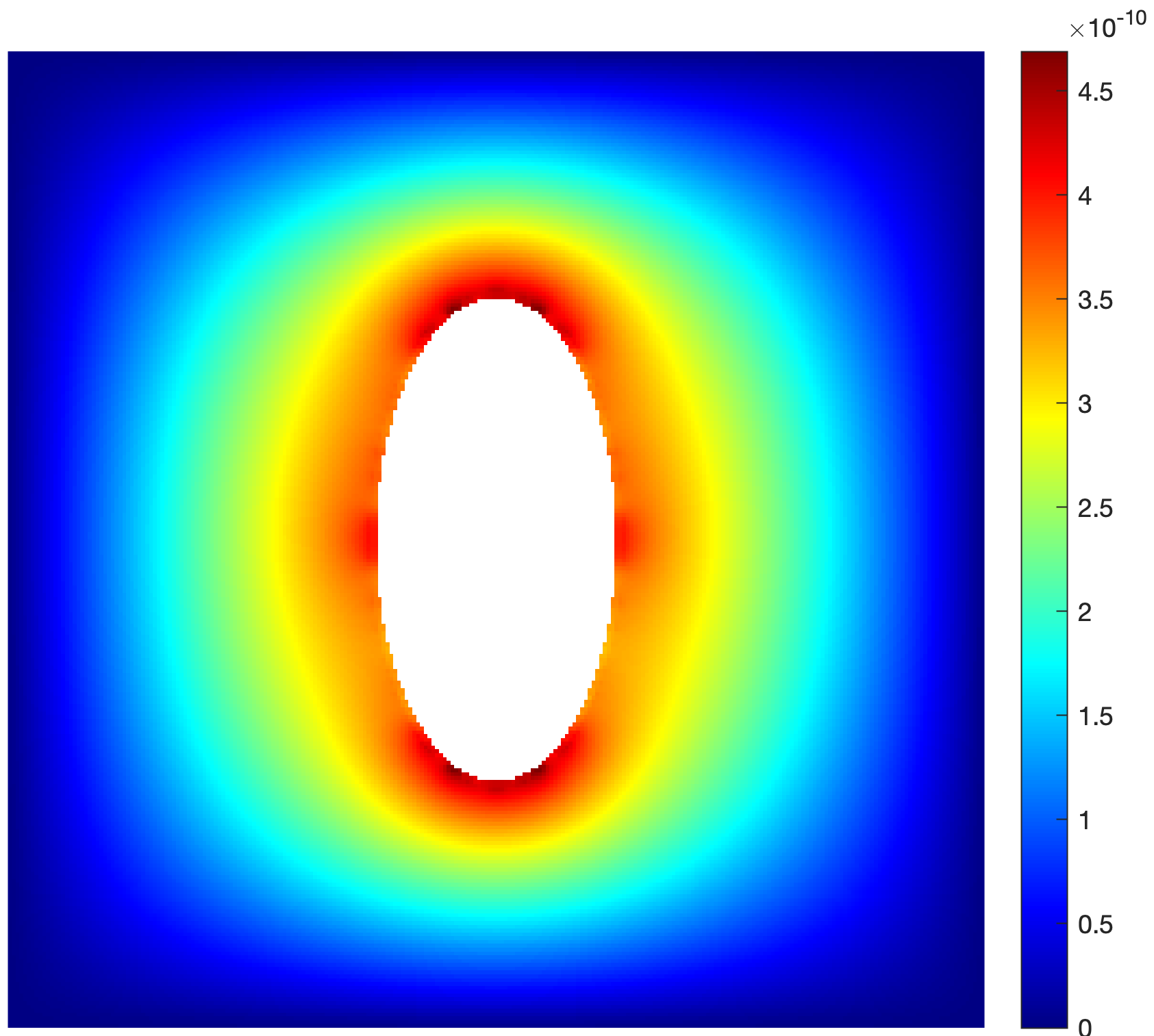}
  \caption{Neumann condition on the irregular boundary. }
  \label{fig:errorNeumannOfProblem2} 
\end{subfigure}
\hspace{20pt}
\begin{subfigure}[b]{.45\linewidth}
  \centering
  \includegraphics[height=1.0\linewidth]{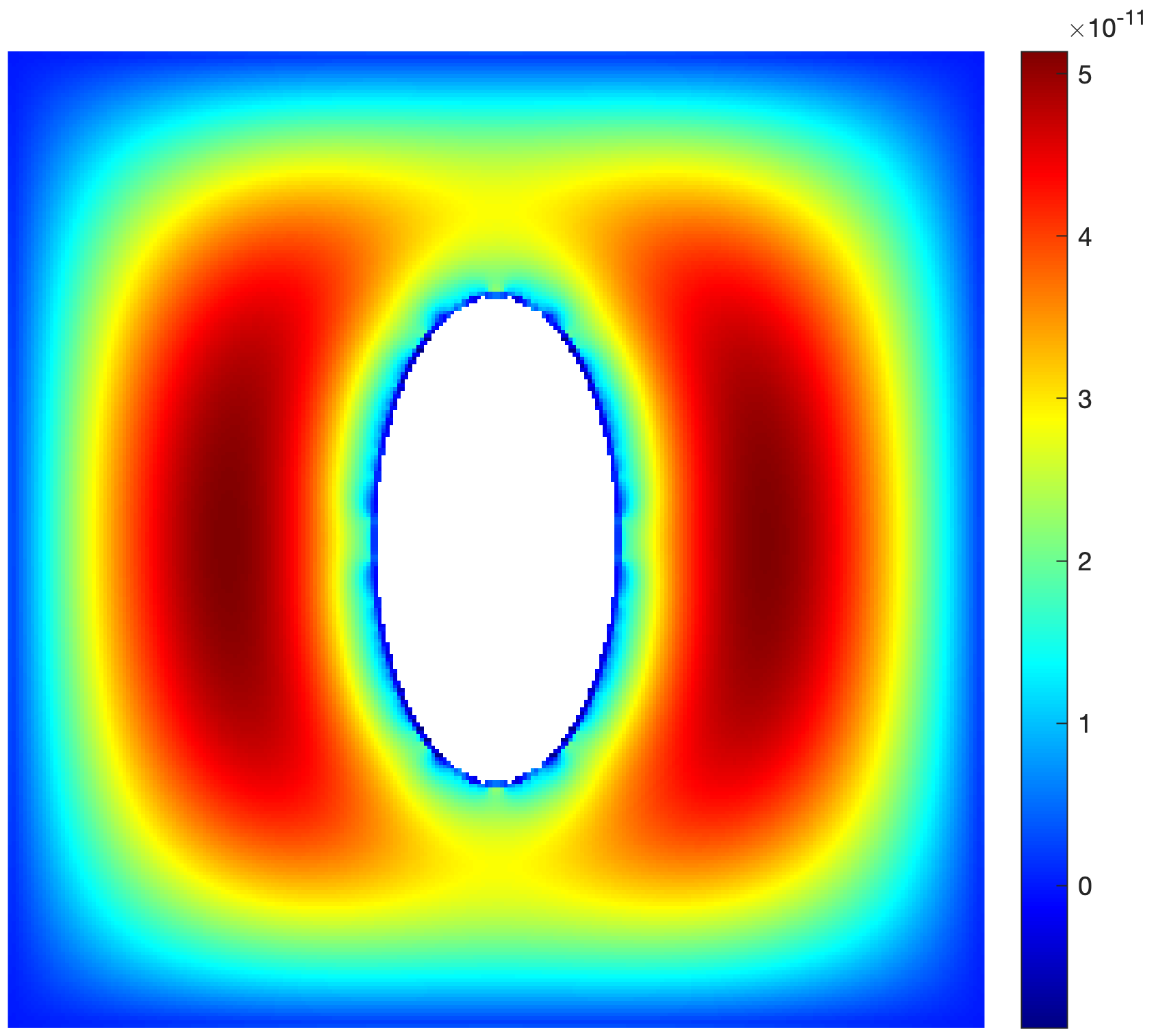}
  \caption{Dirichlet condition on the irregular boundary. }
  \label{fig:errorDirichletOfProblem2}
\end{subfigure}
\caption{Solution errors of Problem 2 with $h=1/256$. 
The Dirichlet condition is applied on the rectilinear sides. }
\end{figure}

\textbf{Problem 3.} 

Consider solving the constant-coefficient elliptic equation (\ref{eq:ccEllipticEq})
with different combinations of $(a,b,c)$ and $\Omega$:
\begin{enumerate}[label={(\roman*)}, leftmargin=*]
\item $\Omega$ is the unit box $[0,1]^2$ and $(a,b,c) = (1, 0, 2)$.
\label{testCase:unrotated}
\item $\Omega$ is the unit box rotated counterclockwise by $\pi/6$. 
In this domain, a system equivalent to case~\ref{testCase:unrotated} would have
$\left(a, b, c\right) = \left(5/4, -\sqrt{3}/2, 7/4\right)$.
\label{testCase:rotated}
\label{testCase:rotatedRoundCorners}
\end{enumerate}

For case~\ref{testCase:unrotated},
the regular geometry and the absence of cross-derivative terms
permits the usage of one-dimensional FD formula. 
In contrast, 
for case~\ref{testCase:rotated} 
we invoke the PLG discretization near the irregular boundary 
to approximate the elliptic operator (\ref{eq:theEllipticOperator})
as a whole.
For case~\ref{testCase:unrotated},
the exact solution is given by
\begin{equation}
u(x_1, x_2) = \sin(4 x_1) \cos(3 x_2). 
\end{equation}
For the other case,
the exact solution is rotated accordingly.
Both test cases assume Dirichlet boundary condition.

The solution errors of both test cases are listed in Table~\ref{tab:errorsOfBox}.
The convergence rates in both cases are close to 4 in all norms,
again confirming the accuracy of our method. 
A comparison between case~\ref{testCase:unrotated} and case~\ref{testCase:rotated}
shows that the $L^\infty$ norm of the error on the grid of $h=1/512$ in the latter case
is roughly four times larger than that in the former case. 

\begin{table}[htb]
\centering
\begin{tabular}{cccccccc}
\hline
\multicolumn{8}{c}{Case~\ref{testCase:unrotated}} \\
\hline
       $h$ &       1/64 &       rate &      1/128 &       rate &      1/256 &       rate &      1/512 \\
\hline
$L^\infty$ &   3.68e-08 &       4.00 &   2.30e-09 &       4.00 &   1.44e-10 &       3.98 &   9.10e-12 \\
     $L^1$ &   1.13e-08 &       4.01 &   7.00e-10 &       4.01 &   4.35e-11 &       3.90 &   2.91e-12 \\
     $L^2$ &   1.50e-08 &       4.01 &   9.32e-10 &       4.00 &   5.81e-11 &       3.96 &   3.73e-12 \\
\hline     
\multicolumn{8}{c}{Case~\ref{testCase:rotated}} \\
\hline
       $h$ &       1/64 &       rate &      1/128 &       rate &      1/256 &       rate &      1/512 \\
\hline       
$L^\infty$ &   1.57e-07 &       4.16 &   8.75e-09 &       3.93 &   5.75e-10 &       3.95 &   3.71e-11 \\
     $L^1$ &   4.92e-08 &       4.08 &   2.92e-09 &       3.92 &   1.92e-10 &       3.91 &   1.28e-11 \\
     $L^2$ &   6.15e-08 &       4.07 &   3.67e-09 &       3.90 &   2.47e-10 &       3.91 &   1.64e-11 \\
\hline
\end{tabular}
\caption{Solution errors of Problem 3, 
with Dirichlet condition on the domain boundaries. }
\label{tab:errorsOfBox}
\end{table}%


\subsection{Efficiency}
\label{sec:efficiency}


Here we show the reductions of relative residuals and relative errors
during the solution of Problem 2, 
and compare with their counterparts on the rectangular domain $R$. 
For both tests we use the same exact solution and the same multigrid parameters.
In Figure~\ref{fig:reduction},
the change of relative residual in the irregular domain $\Omega$ 
(in the rectangular domain $R$, respectively)
is plotted against the number of FMG cycle iterations
in solid lines (in dashed lines, respectively).
In both domains, 
the relative residuals are reduced by a factor of 11.3 per iteration
before their stagnation at the 9th iteration.
The relative errors, 
also shown in Figure~\ref{fig:reduction},
are reduced by a similar factor to that of the relative residuals. 
They stop to descend at the 7th iteration, 
implying that the algebraic accuracy of the solutions is reached. 
To summarize, the multigrid algorithm on irregular domains
behave very much the same as it does on rectangular domains. 

\begin{figure}[htbp]
   \centering
   \includegraphics[width=0.6\linewidth]{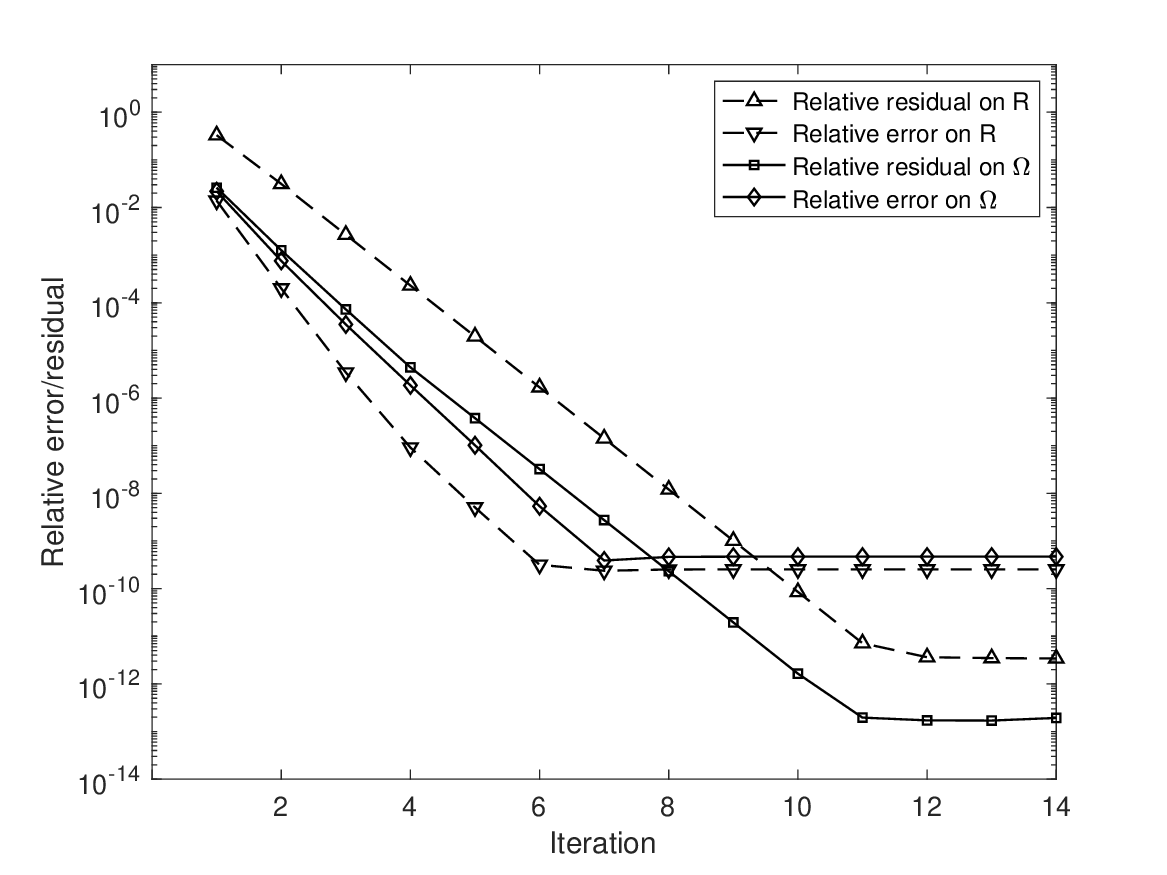}
   \caption{The reduction of relative residuals
   and relative errors during the solution of Problem 2
   in the irregular domain $\Omega$ and in the rectangular domain $R$. 
   The grid is $h=1/128$. 
   The initial guess is the zero function. 
   The multigrid parameters are $\omega = 0.5$ and $\nu_1 = \nu_2 = 3$.
   The ordinates are the relative residuals (relative errors) in $L^\infty$ norm, 
   the abscissa is the number of FMG cycle iterations. 
   The solid lines represent the results in the irregular domain
   and the dashed lines represent the results in the rectangular domain.
   }
   \label{fig:reduction}
\end{figure}

In Table~\ref{tab:timing_1} and \ref{tab:timing_2}, 
we report the time consumption of each component of the solver. 
The column ``cut-cell generation'' refers to the generation of cut cells
and the cell merging process; see Section \ref{sec:grid}. 
The column ``factorization of $L_{22}$'' is self-evident. 
The column ``poised lattice generation'' comprises 
the execution of the PLG algorithm (Definition \ref{def:PLG}) over the irregular cells.  
The column ``operator discretization'' refers to the discretization procedure in Section~\ref{sec:discretizationOfSpatialOp}, 
and finally the column ``multigrid solution'' refers solely to the time consumption 
of FMG cycles.
It can be readiliy observed from the Tables that,
the ``cut-cell generation'' and the ``multigrid solution'' grow quadratically with respect to the mesh size ($\sim h^{-2}$), while the other components only grow linearly ($\sim h^{-1}$). 
This confirms our anlysis in Section~\ref{sec:complexity}, 
and we conclude that the proposed multigrid algorithm is efficient
in terms of solving elliptic equations in complex geometries.

\begin{table}[htbp]
	\centering
	\begin{tabular}{c|ccccc}
		\hline
        $h$ & \makecell{cut-cell\\generation}
        & \makecell{factorization\\of $L_{22}$} 
        & \makecell{poised lattice\\generation} 
        & \makecell{operator\\discretization}
        & \makecell{multigrid\\solution} \\
	    \hline 
	    1/128 & 1.85e$-$03 & 6.59e$-$04 & 2.53e$-$03 & 1.45e$-$02 & 4.03e$-$02 \\
	    1/256 & 6.31e$-$03 & 1.36e$-$03 & 5.17e$-$03 & 2.49e$-$02 & 8.20e$-$02 \\
	    1/512 & 4.43e$-$02 & 2.67e$-$03 & 1.14e$-$02 & 4.58e$-$02 & 2.57e$-$01 \\
	    1/1024 & 1.87e$-$01 & 5.30e$-$03 & 2.70e$-$02 & 9.01e$-$02 & 1.03e$-$00 \\
		\hline 
	\end{tabular}
	\caption{
    Time consumption (in seconds) of solving Problem 2
    on an AMD Threadripper PRO 3975WX at 4.0Ghz with DDR4 2133MHz memory. 
    The sparse factorization of the $L_{22}$ sub-block is computed by
    \texttt{Eigen::SparseLU}.}
	\label{tab:timing_1}
\end{table}

\begin{table}[htbp]
    \centering
    \begin{tabular}{c|ccccc}
		\hline
        $h$ & \makecell{cut-cell\\generation}
        & \makecell{factorization\\of $L_{22}$} 
        & \makecell{poised lattice\\generation} 
        & \makecell{operator\\discretization}
        & \makecell{multigrid\\solution} \\
	    \hline 
	    1/128 & 6.27e$-$03 & 2.01e$-$03 & 9.18e$-$03 & 2.16e$-$01 & 9.23e$-$02 \\
	    1/256 & 2.74e$-$02 & 3.59e$-$03 & 1.90e$-$02 & 4.36e$-$01 & 2.77e$-$01 \\
	    1/512 & 1.06e$-$01 & 8.01e$-$03 & 4.09e$-$02 & 8.75e$-$01 & 9.60e$-$01 \\
	    1/1024 & 4.05e$-$01 & 1.91e$-$02 & 9.21e$-$02 & 1.75e$+$00 & 3.71e$+$00 \\
		\hline 
    \end{tabular}
    \caption{
    Time consumption (in seconds) of solving Problem 3. 
    }
    \label{tab:timing_2}
\end{table}%

\section{Conclusions}
\label{sec:conclusions}

We have proposed a fast fourth-order cut cell method
for solving constant-coefficient elliptic equations in two-dimensional irregular domains. 
For spatial discretization, 
we employ the PLG algorithm to generate finite-volume interpolation stencils
near the irregular boundary.
We then derive the high-order approximation to the elliptic operators
from weighted least squares fitting.
We design a multigrid algorithm for solving the resulting linear system with optimal complexity.
We demonstrate the accuracy and efficiency of our method
by various numerical tests. 

Prospects for future research are as follows.
We expect a straightforward extension of this work
should yield a sixth- and higher-order solver for variable-coefficient elliptic equations.
We also plan to develop a fourth-order INSE solver in irregular domains
based on the GePUP formulation~\cite{Zhang2016:GePUP},
where the pressure-Poisson equation (PPE) and the Helmholtz equations
will be solved by the proposed elliptic solver.

\section*{Acknowledgements}
{
This work was supported by grants with approval number XXXXXXXX from XXX. 
}


\bibliographystyle{siamplain}
\bibliography{bib/ref}

\begin{thebibliography}{10}

\bibitem{Almgren1994}
{\sc A.~S. Almgren, J.~B. Bell, P.~Colella, and T.~Marthaler}, {\em A
  {Cartesian} grid projection method for the incompressible euler equations in
  complex geometries}, SIAM J. Sci. Comput., 18 (1994), pp.~1289--1309.

\bibitem{berger1991}
{\sc M.~J. Berger and R.~J. LeVeque}, {\em A rotated difference scheme for
  {Cartesian} grids in complex geometries}, AIAA,  (1991), pp.~1--9.

\bibitem{Brandt:Multigrid_Techniques}
{\sc A.~Brandt and O.~E. Livne}, {\em Multigrid Techniques}, Society for
  Industrial and Applied Mathematics, 2011.

\bibitem{Briggs:A_Multigrid_Tutorial}
{\sc W.~L. Briggs, V.~E. Henson, and S.~F. McCormick}, {\em A Multigrid
  Tutorial, Second Edition}, Society for Industrial and Applied Mathematics,
  second~ed., 2000.

\bibitem{Brown2001}
{\sc D.~L. Brown, R.~Cortez, and M.~L. Minion}, {\em Accurate projection
  methods for the incompressible {Navier–Stokes} equations}, J. Comput.
  Phys., 168 (2001), pp.~464 -- 499.

\bibitem{Devendran2017}
{\sc D.~Devendran, D.~Graves, H.~Johansen, and T.~Ligocki}, {\em A fourth-order
  {Cartesian} grid embedded boundary method for {Poisson's} equation}, Commun.
  Appl. Math. Comput. Sci., 12 (2017), pp.~51 -- 79.

\bibitem{Gibou2002}
{\sc F.~Gibou, R.~P. Fedkiw, L.-T. Cheng, and M.~Kang}, {\em A
  second-order-accurate symmetric discretization of the {Poisson} equation on
  irregular domains}, J. Comput. Phys., 176 (2002), pp.~205 -- 227.

\bibitem{Givant:BooleanAlgebra}
{\sc S.~Givant and P.~Halmos}, {\em Introduction to Boolean Algebras},
  Springer-Verlag New York, 2009.

\bibitem{Golub}
{\sc G.~H. Golub and C.~F. Van~Loan}, {\em Matrix Computations}, The Johns
  Hopkins University Press, fourth~ed., 2013.

\bibitem{Ji2010}
{\sc H.~Ji, F.-S. Lien, and E.~Yee}, {\em Numerical simulation of detonation
  using an adaptive {Cartesian} cut-cell method combined with a cell-merging
  technique}, Comput. Fluids, 39 (2010), pp.~1041--1057.

\bibitem{Johansen1998}
{\sc H.~Johansen and P.~Colella}, {\em A {Cartesian} grid embedded boundary
  method for {Poisson}'s equation on irregular domains}, J. Comput. Phys., 147
  (1998), pp.~60 -- 85.

\bibitem{Johnston2004}
{\sc H.~Johnston and J.-G. Liu}, {\em Accurate, stable and efficient
  {Navier–Stokes} solvers based on explicit treatment of the pressure term},
  J. Comput. Phys., 199 (2004), pp.~221 -- 259.

\bibitem{Kirkpatrick2003}
{\sc M.~Kirkpatrick, S.~Armfield, and J.~Kent}, {\em A representation of curved
  boundaries for the solution of the {Navier–Stokes} equations on a staggered
  three-dimensional {Cartesian} grid}, J. Comput. Phys., 184 (2003), pp.~1 --
  36.

\bibitem{LeVeque1994}
{\sc R.~J. LeVeque and Z.~Li}, {\em The immersed interface method for elliptic
  equations with discontinuous coefficients and singular sources}, SIAM J.
  Numer. Anal., 31 (1994), pp.~1019--1044.

\bibitem{Li1998}
{\sc Z.~Li}, {\em A fast iterative algorithm for elliptic interface problems},
  SIAM J. Numer. Anal., 35 (1998), pp.~230--254.

\bibitem{Li2003}
{\sc Z.~Li and C.~Wang}, {\em A fast finite differenc method for solving
  {Navier-Stokes} equations on irregular domains}, Commun. Math. Sci., 1
  (2003), pp.~180--196.

\bibitem{Linnick2005}
{\sc M.~N. Linnick and H.~F. Fasel}, {\em A high-order immersed interface
  method for simulating unsteady incompressible flows on irregular domains}, J.
  Comput. Phys., 204 (2005), pp.~157 -- 192.

\bibitem{Liu07}
{\sc J.-G. Liu, J.~Liu, and R.~L. Pego}, {\em Stability and convergence of
  efficient {Navier-Stokes} solvers via a commutator estimate}, Commun. Pure
  Appl. Math., 60 (2007), pp.~1443--1487.

\bibitem{Liu2003}
{\sc T.~Liu, B.~Khoo, and K.~Yeo}, {\em Ghost fluid method for strong shock
  impacting on material interface}, J. Comput. Phys., 190 (2003), pp.~651--681.

\bibitem{Liu2000}
{\sc X.-D. Liu, R.~P. Fedkiw, and M.~Kang}, {\em A boundary condition capturing
  method for {Poisson}'s equation on irregular domains}, J. Comput. Phys., 160
  (2000), pp.~151 -- 178.

\bibitem{Liu1994}
{\sc X.-D. Liu, S.~Osher, and T.~Chan}, {\em Weighted essentially
  non-oscillatory schemes}, J. Comput. Phys., 115 (1994), pp.~200--212.

\bibitem{McCorquodale2001}
{\sc P.~McCorquodale, P.~Colella, and H.~Johansen}, {\em A {Cartesian} grid
  embedded boundary method for the heat equation on irregular domains}, J.
  Comput. Phys., 173 (2001), pp.~620 -- 635.

\bibitem{Morinishi1998}
{\sc Y.~Morinishi, T.~Lund, O.~Vasilyev, and P.~Moin}, {\em Fully conservative
  higher order finite difference schemes for incompressible flow}, J. Comput.
  Phys., 143 (1998), pp.~90 -- 124.

\bibitem{Schwartz2006}
{\sc P.~Schwartz, M.~Barad, P.~Colella, and T.~Ligocki}, {\em A {Cartesian}
  grid embedded boundary method for the heat equation and {Poisson}’s
  equation in three dimensions}, J. Comput. Phys., 211 (2006), pp.~531 -- 550.

\bibitem{Shu1988}
{\sc C.-W. Shu and S.~Osher}, {\em Efficient implementation of essentially
  non-oscillatory shock-capturing schemes}, J. Comput. Phys., 77 (1988),
  pp.~439--471.

\bibitem{Sommariva}
{\sc A.~Sommariva and M.~Vianello}, {\em {Gauss–Green} cubature and moment
  computation over arbitrary geometries}, J. Comput. Appl. Math., 231 (2009),
  pp.~886 -- 896.

\bibitem{Trebotich15}
{\sc D.~Trebotich and D.~T. Graves}, {\em An adaptive finite volume method for
  the incompressible {Navier–Stokes} equations in complex geometries},
  Commun. Appl. Math. Comput. Sci., 10 (2015), pp.~43--82.

\bibitem{LiuMGFM2023}
{\sc L.~Xu and T.~Liu}, {\em Ghost-fluid-based sharp interface methods for
  multi-material dynamics: A review}, Commun. Comput. Phys., 34 (2023),
  pp.~563--612.

\bibitem{Zhang2014}
{\sc Q.~Zhang}, {\em A fourth-order approximate projection method for the
  incompressible {Navier–Stokes} equations on locally-refined periodic
  domains}, Appl. Numer. Math., 77 (2014), pp.~16 -- 30.

\bibitem{Zhang2016:GePUP}
{\sc Q.~Zhang}, {\em {GePUP}: Generic projection and unconstrained {PPE} for
  fourth-order solutions of the incompressible {Navier–Stokes} equations with
  no-slip boundary conditions}, J. Sci. Comput., 67 (2016), pp.~1134--1180.

\bibitem{Zhang2020:YinSets}
{\sc Q.~Zhang and Z.~Li}, {\em Boolean algebra of two-dimensional continua with
  arbitrarily complex topology}, Math. Comput., 89 (2020), pp.~2333--2364.

\bibitem{Zhang:PLG}
{\sc Q.~Zhang, Y.~Zhu, and Z.~Li}, {\em An {AI-aided} algorithm for
  multivariate polynomial reconstruction on {Cartesian} grids and the {PLG}
  finite difference method}.
\newblock Submitted.

\end{thebibliography}

\end{document}